\documentclass[12pt]{amsart}
 
\usepackage{amsmath}
\usepackage{hyperref} 
\usepackage{amsfonts,   amssymb, amsthm, soul,  enumitem} 
\usepackage[letterpaper, margin=1in]{geometry}
\allowdisplaybreaks
\hypersetup{
    colorlinks,
    citecolor=black,
    filecolor=black,
    linkcolor=black,
    urlcolor=black
}
\newtheorem{lemma}{Lemma}[section]
\newtheorem{theorem}[lemma]{Theorem}
\newtheorem{remark}[lemma]{Remark}
\newtheorem{corollary}[lemma]{Corollary}
\newtheorem{definition}[lemma]{Definition}
\newtheorem{proposition}[lemma]{Proposition}
\newcommand{\proj}{\text{proj}}

\newcommand{\nc}{\newcommand}
\nc{\les}{\lesssim}
\nc{\nit}{\noindent}
\nc{\nn}{\nonumber}
\nc{\D}{\partial}
\nc{\diff}[2]{\frac{d #1}{d #2}}
\nc{\diffn}[3]{\frac{d^{#3} #1}{d {#2}^{#3}}}
\nc{\pdiff}[2]{\frac{\partial #1}{\partial #2}}
\nc{\pdiffn}[3]{\frac{\partial^{#3} #1}{\partial{#2}^{#3}}}
\nc{\abs}[1] {\lvert #1 \rvert}
\nc{\cAc}{{\cal A}_c}
\nc{\cE}{{\cal E}}
\nc{\cF}{{\cal F}}
\nc{\cP}{{\cal P}}
\nc{\cV}{{\cal V}}
\nc{\cQ}{{\cal Q}}
\nc{\cGin}{{\cal G}_{\rm in}}
\nc{\cGout}{{\cal G}_{\rm out}}
\nc{\cO}{{\cal O}}
\nc{\Lav}{{\cal L}_{\rm av}}
\nc{\cL}{{\cal L}}
\nc{\cB}{{\cal B}}
\nc{\cZ}{{\cal Z}}
\nc{\cR}{{\cal R}}
\nc{\cT}{{\cal T}}
\nc{\cY}{{\cal Y}}
\nc{\cX}{{\cal X}}
\nc{\cXT}{{{\cal X}(T)}}
\nc{\cBT}{{{\cal B}(T)}}
\nc{\vD}{{\vec \mathcal{D}}}
\nc{\efield}{\mathcal{E}}
\nc{\vE}{{\vec \efield}}
\nc{\vB}{{\vec \mathcal{B}}}
\nc{\vH}{{\vec \mathcal{H}}}
\nc{\mR}{\mathcal R}
\nc{\mF}{\mathcal F}
\nc{\mE}{\mathcal E}
\nc{\ty}{{\tilde y}}
\nc{\tu}{{\tilde u}}
\nc{\tV}{{\tilde V}}
\nc{\Pc}{{\bf P_c}}
\nc{\bx}{{\bf x}}
\nc{\bX}{{\bf X}}
\nc{\bXYZ}{{\bf XYZ}}
\nc{\bY}{{\bf Y}}
\nc{\bF}{{\bf F}}
\nc{\bS}{{\bf S}}
\nc{\dV}{{\delta V}}
\nc{\dE}{{\delta E}}
\nc{\TT}{{\Theta}}
\nc{\dPsi}{{\delta\Psi}}
\nc{\order}{{\cal O}}
\nc{\Rout}{R_{\rm out}}
\nc{\eplus}{e_+}
\nc{\eminus}{e_-}
\nc{\epm}{e_\pm}
\nc{\sgn}{\text{sgn}}
\nc{\eps}{\varepsilon}
\nc{\vnabla}{{\vec\nabla}}
\nc{\G}{\Gamma}
\nc{\w}{\omega}
\nc{\mh}{h}
\nc{\mg}{g}
\nc{\vphi}{\varphi}
\nc{\tlambda}{\tilde\lambda}
\nc{\be}{\begin{equation}}
\nc{\ee}{\end{equation}}
\nc{\ba}{\begin{eqnarray}}
\nc{\ea}{\end{eqnarray}}
\renewcommand{\k}{\kappa}

\nc{\g}{\gamma}
\nc{\ol}{\overline}
\newcommand{\n}{\nu}
 
\nc{\pT}{\partial_T}
\nc{\pz}{\partial_z}
\nc{\pt}{\partial_t}
\nc{\la}{\langle}
\nc{\ra}{\rangle}
\nc{\infint}{\int_{-\infty}^{\infty}}
\nc{\halfwidth}{6.5cm}
\nc{\figwidth}{10cm}

\nc{\nlayers}{L} \nc{\nsectors}{M}
\nc{\indicator}{\mathbf{1}}
\nc{\Rhole}{R_{\rm hole}}
\nc{\Rring}{R_{\rm ring}}
\nc{\neff}{n_{\rm eff}}
\nc{\Frem}{F_{\rm rem}}
\nc{\R}{\mathbb R}
\nc{\mJ}{\mathcal J}
\nc{\T}{\mathbb T}
\nc{\C}{\mathbb C}
\nc{\Z}{\mathbb Z}
\nc{\N}{\mathbb N}
\nc{\DD}{\Delta}
\nc{\cD}{\mathcal D}
\nc{\lnorm}{\left\|}
\nc{\rnorm}{\right\|}
\nc{\rnormp}{\right\|_{\ell^{p,\eps}}}
\nc{\rar}{\rightarrow} 
\allowdisplaybreaks
\title[Talbot Effect on the Sphere and Torus]{Talbot Effect on the Sphere and Torus for $d\geq 2$}
\author[ Erdo{\u g}an, Huynh,  McConnell]{M.~Burak~Erdo{\u g}an, Chi~N.~Y.~Huynh, Ryan~McConnell}
\thanks{ The first author was partially supported by the NSF grant  DMS-2154031 and Simons Foundation Grant 634269. The second and third authors are partially supported by the NSF grant  DMS-2154031. Some of the results of this paper already appeared in the thesis of the third author, \cite{huynh2022study}.}
	\address{Department of Mathematics \\
		University of Illinois \\
		Urbana, IL 61801, U.S.A.}
	\email{berdogan@illinois.edu}
		\email{huynhngocyenchi@gmail.com}
			\email{ryanm12@illinois.edu}
\begin{document}
\begin{abstract}
We utilize exponential sum techniques to obtain upper and lower bounds for the fractal dimension of the graph of solutions to the linear Schr\"odinger equation on $\mathbb{S}^d$ and $\mathbb{T}^d$. Specifically for $\mathbb S^d$, we provide   dimension bounds using both $L^p$ estimates of Littlewood-Paley blocks, as well as assumptions on the Fourier coefficients.   In the appendix, we present a slight improvement to the bilinear Strichartz estimate on $\mathbb{S}^2$ for functions supported on the zonal harmonics. We apply this to demonstrate an improved local well-posedness result for the zonal cubic NLS when $d=2$, and a nonlinear smoothing estimate when $d\geq 2$. As a corollary of the nonlinear smoothing for solutions to the zonal cubic NLS, we find dimension bounds generalizing the results of \cite{ErTz2} for solutions to the cubic NLS on $\T$. Additionally, we obtain several results on $\mathbb{T}^d$ generalizing the results of the $d=1$ case.
\end{abstract}
\maketitle
\tableofcontents
\section{Introduction}

In this paper, we investigate so called the Talbot effect and fractal solutions of  the linear Schr\"{o}dinger equation on certain compact manifolds $\mathcal{M} = \mathbb{S}^d,$ or $\mathbb{T}^d$ for $d \geq 2$:
\begin{align}\label{Equation: Linear Schrodinger}
 \begin{cases} 
       iu_t + \bigtriangleup_\mathcal{M} u = 0\\
       u(0,x) = f(x) \in L^2(\mathcal{M}), 
   \end{cases} 
\end{align}
where $\bigtriangleup_\mathcal{M}$ is the Laplace-Beltrami operator on $\mathcal{M}$.

Many authors have studied the properties and dimension of the graph of the solution to the linear Schr\"{o}dinger equation and other dispersive equiations on $\mathbb{T}$ with varying initial data; see, e.g.,  \cite{Be, BeK1, BMS,  Os1, KaRo, Ro, ErTz1, ErTz2,OC13,chenolv,HV,cet,V,BurakNikos,Ol,OlTs}.  The history of this line of inquiry starts with an optical experiment in 1836 where Talbot studied monochromatic light passing through a diffraction grating \cite{Ta}. He observed there is a certain distance (now called the Talbot distance) at which the diffraction pattern reproduces the grating pattern. It was further remarked that the pattern appears to be a finite linear combination of the grating pattern at each rational multiple of the Talbot distance. This phenomenon has since been referred to as the Talbot effect. Berry and collaborators were among the first to carry out exact calculations and numerical works on the Talbot effect in \cite{Be, BeK1, BeLe}. In particular, in \cite{BeK1} it was proved that at rational times the solution is a linear combination of finitely many translates of the initial data with the coefficients being Gauss sums, also see \cite{Ta1,Ta2,Ol,OlTs}.  This phenomenon is often called {\it quantization} in the literature. In \cite{BeK1}, the authors also observed that the solution at irrational times has a fractal  profile. 
In particular for step function initial data at rational times one observes a step function,  and a continuous but nowhere differentiable function with fractal dimension $\frac32$ at irrational times\footnote{Recall that the fractal dimension (or upper Minkowski/box dimension) of a bounded set $E$ is given by $\overline{\dim}E:=\limsup_{\epsilon \to 0} \frac{\log \mathcal{N}(E,\epsilon)}{\log(1/\epsilon)}$, where $\mathcal{N}(E,\epsilon)$ is the minimum number of boxes of sidelength $\epsilon$ needed to cover $E$.}.   Finally, it was conjectured that this phenomenon should   occur in higher dimensions and even when there is a nonlinear perturbation, also see \cite{ZWZX} for an experimental justification.

 In the field of mathematics, Oskolkov proved in \cite[Proposition 14]{Os1} that for bounded variation initial data, the solution of any dispersive PDE on $\mathbb{T}$ with polynomial dispersion relation is a continuous function of $x$ at irrational times. Kapitanski and Rodniaski showed in \cite{KaRo} the solution to the linear Schr\"{o}dinger equation at irrational times is more regular in the Besov scale than at rational times. Rodniaski then used results in \cite{KaRo} to justify Berry's conjecture, proving, that given initial data in $BV(\mathbb{T}) \setminus H^{1/2+}(\mathbb{T})$, the graph of the real and imaginary parts of the solution to \eqref{Equation: Linear Schrodinger} has fractal dimension $3/2$ for almost every time. 

This paper is motivated by the works by the first author, Tzirakis, and Shakan  in \cite{ErTz1, ErTz2, BurakNikos,ErdShak}. 
In these papers, the authors considered the case of polynomial and nonpolynomial dispersion relations, and proved several results on dimension bounds using exponential sum estimates on $\mathbb{T}$.  In particular, they obtained fractal dimension bounds for the graph of the real and imaginary part of solutions with bounded variation initial data. In addition, the results on the dimension were extended to nonlinear counterparts via nonlinear smoothing estimates.
As these works were done on $\mathbb{T}$, it is clearly desirable to obtain analogous estimates on more general compact manifolds. In the case of $\T^d$, one can easily extend the rational time quantization results on $\T$   to $\T^d$, \cite{Ta1}.  It is also clear that  the results in \cite{Os1,Ro, ErTz1, ErTz2, BurakNikos,ErdShak} can be extended  to $\T^d$ in the case when the initial data is a tensor function, establishing the existence of fractal solutions, and a dichotomy similar to the one on $\T$. On the other hand, extending the fractal behavior results to even just $\mathbb{S}^{2}$ or to more general functions on $\mathbb T^d$ is not as straightforward, and proving satisfactory dimension bounds on $\mathbb{S}^{d}$ for $d \geq 3$ turns out to be quite challenging. We note that in \cite{Ta2}, Taylor studied the Talbot effect for the Schr\"odinger propagator on $\mathbb S^{d}$ and obtained multiplier estimates at rational times. This is analogous to the quantization behavior at rational times on the torus since finite linear combinations of translations are bounded operators on all $L^p$ spaces. Also see \cite{MR2439212} and \cite{chamizo2023quantum} for various results on the quantization on $\mathbb S^2$.  In this paper, by establishing the existence of fractal solutions at irrational times, we obtain a dichotomy on $\mathbb S^d$ similar to the one on $\T$, or $\T^d$.

To study this problem, as in the earlier papers,  we will utilize Besov spaces,   $B^s_{p,\infty}$, defined by the norm 
 $\|f\|_{B^\gamma_{p,\infty}}:=\sup\{N^{\gamma} \|  f_N\|_{L^p}: N\geq 1, \text{ dyadic}\},$ $1\leq p\leq \infty,$ 
where $f_N$ is the Littlewood-Paley projection to frequencies $\approx N$.
 Recall that for $0<\gamma<1$, $C^\gamma(\T)$ coincides with $B^\gamma_{\infty,\infty}(\T)$,   and that  if $f:\T\to\R$ is in $C^\gamma$, then
the graph of $f$ has fractal  dimension $D\leq 2-\gamma$. For lower bounds we have the following result of Deliu and Jawerth \cite{DeJa} (also see \cite[Theorem 2.24]{BurakNikos}):  Fix $\gamma\in [0,1]$. The graph of a continuous function $f:\T\to\R$  has fractal  dimension $D \geq 2-\gamma$ provided that $f\not\in B^{\gamma}_{1,\infty}$. Analogous results hold for $\T^d$ and $\mathbb S^d$;  see Theorem~\ref{upperS} and Theorem~\ref{Spherical Deliu Jawerth}  below for the case of $\mathbb S^d$.

Before moving on to the statements of the results, we first take a moment to discuss generalizations of the $BV(\mathbb{T})$ requirement of \cite{Os1, Ro}. Specifically, $f\in BV(\mathbb{T})$ implies that $\widehat{f}(n)=\frac1{in}\widehat{df}(n)$, which leads to, for $1\leq p\leq 2$,  $\|f_N\|_{L^p_x(\mathbb{T})}\lesssim N^{-\frac{1}{p}}$.   It follows then that a natural generalization would be the requirement that $\|f_N\|_{L^p}\lesssim N^{-\frac{d}{2}-s}$ for some $p\geq 1$, and $s\geq 0$. This is the approach taken for Theorems \ref{Theorem: General Sphere Theorem} and \ref{Theorem: General TOrus}.  On the other hand, $f\in BV(\mathbb{T})$ also implies that $|\widehat{f}(n)|\lesssim n^{-1}$ and additional bound on the differences up to a phase.  
Thus, the next possible generalization is to require decay on the Fourier coefficients,  
which is the approach of Theorems \ref{Theorem: Zonal Bound}, \ref{Theorem: Gaussian Beams}, and \ref{Theorem: Td L2}. 

In particular, when $\mathcal{M}$ is understood, we define $\dim_t(f)$ to be the maximum of the fractal dimensions of the real and imaginary parts of $u(\cdot, t)$, the solution to \eqref{Equation: Linear Schrodinger} emanating from $f$ at time $t$; see \eqref{eq:SchProp} for $\mathbb S^d$. Using this notation, we show in Section~\ref{sec:TalbotSphere} bounds on $\dim_t(f)$ depending on the $L^p$ norms of the Littlewood-Paley pieces, $f_N$, of $f$ on $\mathbb S^d$.

\begin{theorem}\label{Theorem: General Sphere Theorem}
Let $f: \mathbb{S}^{d} \to \mathbb{R}$, $1 < p \leq 2$, $s > \max\{\frac{d}{p}-\frac{d+1}{2}, 0\}$.  Assume that  $\|f_N\|_{L^p} \lesssim N^{-\left(\frac{d}{2}+s\right)}$.  
Then for almost all $t$, the solution  $u(\cdot, t)$ to \eqref{Equation: Linear Schrodinger} is in  
$ C^{\gamma-}$ for $\gamma = \min\{s, s +\frac{d+1}{2}-\frac{d}{p}, 1\}$. Hence for almost all $t$, 
$\dim_t(f)\leq  d+1 - \gamma$, where $\gamma = \min(s, s +\frac{d+1}{2}-\frac{d}{p}, 1)$.

In the case  $d=2$,  if  $f\not\in H^{s+2-\frac{2}{p}+}(\mathbb{S}^2)$ in addition to the  hypothesis above, then 
\[
\dim_t(f)\geq \max\left(\tfrac{3}{4}+\tfrac{2}{p}-s, 2\right).
\]
\end{theorem} 
We only have the lower bound in the case $d=2$ because it relies on a Strichartz estimate which is not strong enough for this purpose when  $d>2$.
 
 \begin{corollary}Let $p=\frac{2d}{d+1}$ and assume   $\|f_N\|_{L^p} \lesssim N^{-d/p}$.  
Then for almost all $t$,  $u (t,x) \in C^{\frac12-}$, and hence   
$\dim_t(f)\leq d+\frac12$.
 \end{corollary}

As the above formulation is not easy to use for specific $f$, we also seek generalizations in terms of the Fourier coefficients of $f$. In Section \ref{Section: Spherical Specifics}, we specialize to the case $d = 2$ and provide bounds for $\dim_t(f)$ using only information on the Fourier coefficients in the specific cases that $f$ is supported on the Zonal harmonics, $Y_n:=Y_n^0$, in $\mathbb{S}^2$: 

\begin{theorem}\label{Theorem: Zonal Bound}
Let $f(\theta,\phi) = \displaystyle\sum_{n=0}^{\infty} a_n \, Y_n(\theta,\phi)$.  
If for some $1< p < 2$ we have
\[|a_{ n}| \lesssim \tfrac{1}{n^{p}}, \text{ and } \quad |a_{ n} - a_{ n-1}| \lesssim \tfrac{1}{n^{p+1}}\]
for all $n \in \mathbb{N} \cup \{0\}$, then for almost all $t$,  $u(x,t) \in C^{(p-1)-}$, and hence 
$\dim_t(f)\leq 4-p$. 

If, in addition, $f\not\in H^{p-\frac{1}{2}+}(\mathbb{S}^2)$, then we also find $\dim_t(f)\geq \max\left(\tfrac72 - p, 2\right)$.
\end{theorem}
We additionally extend this result to the case when $f$ is supported on $\{Y^k_n\}_n$ for $k$ fixed in Lemma~\ref{Lemma: Genearal fixed k 2d lemma}, and the case when $f$ is supported on Gaussian beams, $\{Y_n^n\}_n$, in Theorem~\ref{Theorem: Gaussian Beams}.

In addition, we also note the following Corollary, whose proof is immediate from the methods in Theorem \ref{Theorem: Zonal Bound} and is stated separately from Theorem \ref{Theorem: Zonal Bound} due to the work in Appendix \ref{Appendix: Strichartz}.

\begin{corollary}\label{Corollary: Zonal In Sd}
Suppose that $\frac{d}{2} < p < \frac{d}{2}+1$ and $\{a_n\}$ satisfies
\[
|a_{ n}| \lesssim \tfrac{1}{n^{p}}, \text{ and } \quad |a_{ n} - a_{ n-1}| \lesssim \tfrac{1}{n^{p+1}}.
\]
Then if $f\in\mathbb{S}^d$ is supported on the zonal harmonics with Fourier coefficients $\{a_n\}_n$, then for almost all $t$, $u\in C^{p-\frac{d}{2}+}$  and $\dim_t(f)\leq (d+1)-\left(p-\frac{d}{2}\right).$
\end{corollary}

 In Appendix \ref{Appendix: Strichartz}, we provide an improved estimate on products of zonal harmonics on $\mathbb{S}^2$ which will be useful in obtaining lower bounds in Section \ref{Section: Spherical Specifics}. As a consequence of this estimate and the symmetries of the cubic nonlinear Schr\"odinger equation, we obtain both local well-posedness to \eqref{Equation: 3nls on sphere} for $s > 0$ in Theorem \ref{Corollary: 2d zonal wellposed} and nonlinear smoothing\footnote{See Section \ref{Section: Smoothing} for an introduction to nonlinear smoothing and motivation for the statement of Theorem \ref{Theorem: nonlinear smoothing}.} in Theorem \ref{Theorem: nonlinear smoothing} for the class of functions supported on the zonal harmonics. 
\begin{theorem}\label{Theorem: nonlinear smoothing}
    Let $d\geq 2$, $s >\frac{d-2}{2}$, and
    \[
    0\leq \varepsilon < \min\left(\tfrac{1}{2}\left(s-\tfrac{d-2}{2}\right), 1\right).
    \]
    For $u_0\in \mathcal{Z}^s(\mathbb{S}^d)$, let $u$ denote the solution to \eqref{Equation: Zonal Cubic NLS} emanating from $u_0$ with local   existence time $T > 0$. Then, letting (see  \eqref{Equation: Phase Rotation Removal} below)
    \begin{equation}
        \gamma(t; u) = \frac{2}{\pi\omega_d}\sum_{k,\ell}\overline{\widehat{u}_{k}}(t)\widehat{u}_{\ell}(t)\int_{0}^\pi Y_{k}(\theta)Y_{\ell}(\theta)\,d\theta,
    \end{equation}
 we have
    \[
    u-e^{it\bigtriangleup_{\mathbb{S}^d}\mp i\int_0^t \gamma(s;\, u)\,ds}u_0\in C^0_t\left([0, T), H^{s+\varepsilon}_x\right).
    \]
    In particular, for all $ 0<t<T$,
    $$
    \inf_{\theta\in\R} \big\|u-e^{i\theta} e^{it\bigtriangleup_{\mathbb{S}^d}} u_0\big\|_{  H^{s+\varepsilon}_x}\les 
    C_{\|u_0\|_{H^s}}.
    $$
    \end{theorem}

We then use this nonlinear smoothing to show our final theorem on the cubic NLS posed on zonal functions on $\mathbb{S}^d$ for $d\geq 2$:

\begin{theorem}\label{Theorem: Cubic NLS dimension bound}
Let $d\geq 2$ and $f$ satisfy the hypothesis of Theorem \ref{Corollary: Zonal In Sd} for some $\frac{d+1}{2} < p < \frac{d+2}{2}$ and let $u$ denote the solution to \eqref{Equation: Zonal Cubic NLS} emanating from $f$. Then $f\in \mathcal{Z}^{p-1/2-} := H^{p-1/2-}(\mathbb{S}^d)\cap    span \{Y_n\,:\,n\in\mathbb{N}\}$ and
\[
\dim_t(f)\leq (d+1)-\left(p-\tfrac{d}{2}\right).
\]
\end{theorem}

Finally, we offer several extensions of the results from \cite{ErdShak} to $\mathbb{T}^d$ in Appendix \ref{Section: Toral Bounds}. Of particular interest is the statement of the result of \cite{Ro} for Vitali BV functions (for definitions and a survey, see \cite{BVPropertiesCorrection, BVProperties, VitaliBV}), and Theorems \ref{Theorem: Torus 2d POlygon} and \ref{Theorem: Torus general Polygon} on the graphs of solutions emenating from characteristic functions of polygons and polytopes.

We finish the introduction with some notation. We say that $f\lesssim g$ if there is a $C > 0$ so that $f\leq Cg$, and also denote $a+$ to be $a+\varepsilon$ for all $\varepsilon > 0$, with implicit constants that will depend on $\varepsilon$. We also define the bracket $\langle\cdot\rangle := (1+|\cdot|^2)^{1/2}.$

\section{Harmonic Analysis on the Sphere}\label{Section: Sphere}
We start with a discussion of spherical convolution of $f:\mathbb S^d\to \C$ and $g:[-1,1]\to \C$, both continuous, say: 
$$
 f*g(x):= \frac{1}{\omega_d}  \int_{\mathbb S^d} f(y) g(\la x,y \ra) \ d\sigma(y),\,\,\,\,x\in \mathbb S^d,
$$ 
where $d\sigma$ is the surface measure and $\omega_{d}$ the surface volume of $\mathbb{S}^{d}$, and $\la x,y\ra$ is the $\R^{d+1}$ inner product of $x,y\in \mathbb S^d$.  
Note that  
$$\|g(\la x,\cdot\ra)\|_{L^r(\mathbb S^d)}^r= \frac{1}{\omega_d}  \int_{\mathbb S^d}|g(\la x,y \ra)|^r \ d\sigma(y)=  \frac{\omega_{d-1}}{\omega_d}   \int_{-1}^1|g(\tau)|^r (1-\tau^2)^{\frac{d-2}2} \ d\tau
$$
is independent of $x\in \mathbb S^d$ (and similarly if we take the norm in the $x$ variable for fixed $y$).
Therefore, we define 
\begin{multline}\label{Equation: wlp norm}
\|g\|_{L^r_w([-1,1])} :=  \|g(\la e_{d+1}, \cdot\ra)\|_{L^r(\mathbb S^d)} \\
= \left(\frac{\omega_{d-1}}{\omega_d} \int_{-1}^1|g(t)|^r (1-t^2)^{\frac{d-2}2} \ dt
 \right)^{1/r} =\left(\frac{\omega_{d-1}}{\omega_d} \int_{0}^\pi|g(\cos(\theta))|^r [\sin(\theta)]^{d-1} \ d\theta\right)^{1/r}.
\end{multline}
With that, and by an application of Holder, Minkowski inequalities and Riesz-Thorin interpolation, one gets,  for  $\frac{1}{p} = \frac{1}{r} + \frac{1}{q} - 1$,  
  \be\label{SphericalYoung}
  \|f * g\|_{L^p(\mathbb S^d)} \leq \|f\|_{L^q(\mathbb S^d)}\|g\|_{L^r_w([-1,1])} . 
  \ee 
For more details see  \cite[Chapter 2]{DX}.

On $\mathbb{S}^d$ (with obvious metric) we denote the Laplace-Beltrami operator  by $\bigtriangleup$. With the inner product 
 $\langle f,g \rangle_{\mathbb{S}^{d}} := \frac{1}{\omega_d} \int_{\mathbb{S}^{d}} f(x) \overline{g(x)} d\sigma(x)$, 
 %where $d\sigma$ is the surface measure and $\omega_{d}$ the surface area of $\mathbb{S}^{d}$, 
 the eigenfunctions (spherical harmonics) of $-\bigtriangleup$  form an orthonormal basis for $L^2(\mathbb{S}^d)$, in particular
$$
f(x)=\sum_{n=0}^\infty \proj_{n}f(x),
$$
where the series converges in $L^2$. 
Here $\proj_n$ is the projection onto $E_{n(n+d-1)}$, the subspace of $L^2(\mathbb{S}^d)$ spanned by the eigenfunctions with eigenvalue $n(n+d-1)$. 
Recall that (see, e.g.,   \cite[Section 1.2]{DX}), these projections are given by a spherical convolution of $f$:
\begin{equation}\label{Equation: Spherical Reproducing Kernel}
\proj_{n}f(x) = f*Z_n (x)= \frac{1}{\omega_d} \int_{\mathbb{S}^d} f(y)Z_n(\langle x, y\rangle)\,d\sigma(y), \,\,\,x\in \mathbb S^d.
\end{equation}
Here,   $Z_n$'s are the zonal  harmonics:   
\begin{align}\label{Equation: General Zonal for Reproducing}
Z_n(\tau) := \big(\tfrac{2n}{d-1} + 1\big) \tfrac{\Gamma(d/2) \Gamma(d + n - 1)}{\Gamma(d) \Gamma(n +d/2)} P_n^{(\frac{d-2}2, \frac{d-2}2)}(\tau),\,\,\,\tau\in[-1,1],
\end{align}
where $P^{(\alpha, \beta)}_n$ denotes the Jacobi polynomial.
To prove the theorems on the sphere, we need to study zonal harmonics in some detail.  In particular, we need to understand the growth and oscillation of the Jacobi polynomials:
\begin{lemma}\label{Lemma: Zonal and convolution bounds}  For $d\geq 2$ the have
\begin{align*}
Z_n(\cos(\theta)) = b_n^+(\theta)e^{in\theta}+b_n^-(\theta)e^{-in\theta}+E_1(\theta,n), \,\,\,\text{ where}
\end{align*} 
\begin{align}\label{Equation: Bounds on Convolution Zonal}
|b_n^\pm(\theta)|\lesssim \frac{n^{d-1}}{\langle n\theta\rangle^\frac{d-1}{2}}, \quad |b_n^\pm(\theta)-b_{n-1}^\pm(\theta)|\lesssim \frac{n^{d-2}}{\langle n\theta\rangle^\frac{d-1}{2}}, \quad \mbox{and}\quad|E_1(\theta,n)|\lesssim n^{d-3}. 
\end{align}
Similarly, we find that the zonal spherical harmonic of degree $n$ satisfies  
\begin{align*}
    Y_n( \theta ) = \mathfrak{b}_n^+(\theta)e^{in\theta}+\mathfrak{b}_n^-(\theta)e^{-in\theta}+E_2(\theta,n), \,\,\,\text{ where}
\end{align*} 
    \begin{align*}
    |\mathfrak{b}_n^\pm(\theta)|\lesssim \frac{n^\frac{d-1}{2}}{\langle n\theta\rangle^\frac{d-1}{2}}, \quad |\mathfrak{b}_n^\pm(\theta)-\mathfrak{b}_{n-1}^\pm(\theta)|\lesssim \frac{n^\frac{d-3}{2}}{\langle n\theta\rangle^\frac{d-1}{2}}, \quad \mbox{and}\quad|E_2(\theta,n)|\lesssim n^{\frac{d-5}{2}}. 
    \end{align*}

    Lastly, let $d = 2$ and fix $k\in\Z$. For $n\gg |k|$, we have the expansion
    \begin{align*}
        Y_n^k( \theta ) = \mathfrak{b}_{n,k}^+(\theta)e^{in\theta+ik\phi}+\mathfrak{b}_{n,k}^-(\theta)e^{-in\theta+ik\phi}+E_{2,k}(\theta,n;k), \,\,\,\text{ where}
    \end{align*} 
          \begin{align*}
        |\mathfrak{b}_{n,k}^\pm(\theta)|\lesssim \frac{n^\frac{1}{2} }{\langle n\theta\rangle^{ \frac{1}{2}}}, \quad |\mathfrak{b}_{n,k}^\pm(\theta)-\mathfrak{b}_{n-1,k}^\pm(\theta)|\lesssim \frac{1 }{n^\frac{1}{2}\langle n\theta\rangle^{ \frac{1}{2}}}, \quad \mbox{and}\quad|E_{2,k}(\theta,n)|\lesssim \frac{1 }{n^\frac{1}{2}\langle n\theta\rangle^{ \frac{1}{2}}}. 
        \end{align*}
\end{lemma}
\begin{proof}
    For all of the above we first note that we may reduce, by symmetry, to considering $[0, \pi/2]$. Equation \eqref{Equation: General Zonal for Reproducing} now allows us to write 
    \begin{align*}
        Z_n(\cos\theta) = c_d(n)P_n^{(\frac{d-2}2, \frac{d-2}2)}(\cos\theta), 
    \end{align*}
    where Stirling's approximation,
\be\label{stir}
    n! = \sqrt{2\pi n}\left(\tfrac{n}{e}\right)^n\left(1+O\left(\tfrac{1}{n}\right)\right),
\ee
 yields $c_d(n) = n^\frac{d}{2}\left(1+O\left(\frac{1}{n}\right)\right)$. The claims of \eqref{Equation: Bounds on Convolution Zonal} will then follow by establishing 
    \begin{equation}\label{eq:Pn}
        P_n^{(\frac{d-2}2, \frac{d-2}2)}(\cos\theta) = \tilde{b}_n^+(\theta)e^{in\theta}+\tilde{b}_n^-(\theta)e^{-in\theta}+\tilde{E_1}(\theta,n),\,\,\,\text{ where}
    \end{equation} 
    \begin{align}\label{Equation: Bound for Zn reduction}
        |\tilde{b}_n^\pm(\theta)|\lesssim \frac{n^\frac{d-2}{2}}{\langle n\theta\rangle^\frac{d-1}{2}}, \quad |\tilde{b}_n^\pm(\theta)-\tilde{b}_{n-1}^\pm(\theta)|\lesssim \frac{n^\frac{d-4}{2}}{\langle n\theta\rangle^\frac{d-1}{2}}, \quad \mbox{and}\quad|\tilde{E_1}(\theta,n)|\lesssim n^{\frac{d-6}{2}}. 
    \end{align}

        In that direction, we note that by \cite[Theorem 1.2]{BarGatError} we have, for $\theta\in[0,\pi/2]$,
        \be\label{Equation: Better Jacobi Asymptotics}
          P^{(\frac{d-2}{2},\frac{d-2}{2})}_n(\cos\theta)= 
             \frac{\Gamma(n+d/2)}{n!}  \left[\alpha_d(\theta)\frac{J_\frac{d-2}{2}(M\theta)}{(M\theta)^\frac{d-2}{2}}+\beta_d(\theta)\frac{J_{d/2}(M\theta)}{(M\theta)^{d/2}}+O\left(\tfrac{\theta^{2 }}{n^{2}}\right)\right]
        \ee
        where $M= n+\frac{d-1}{2}$,  $|\alpha_d(\theta)|\les 1, |\beta_d(\theta)|\les \theta^2$,
        and $J_\gamma$ is the Bessel function of the first kind of order $\gamma$. Noting that Stirling's approximation again gives
        \be\label{stirlapp}
        \frac{\Gamma(n+d/2)}{n!} = n^\frac{d-2}{2}\left(1+O\left(\tfrac{1}{n}\right)\right),
        \ee
        we reduce to finding suitable bounds on $J_{\frac{d-2}{2}}(M\theta)$ and $J_{\frac{d}{2}}(M\theta)$. This, however, follows by clasically relating Bessel functions to Fourier transforms of surface measures associated to spheres. Let $\sigma_{d-1}$ be the surface measure of the $d-1$ dimensional unit sphere,  $d\geq 2$. We have
        $$
  (2\pi)^{\frac{d}{2}}   \frac{J_{\frac{d-2}{2}}( |\xi|)}{|\xi|^{\frac{d-2}{2}}}  =\widehat{\sigma_{d-1}}(\xi)=\int e^{-ix\cdot \xi } \ d\sigma_{d-1}(x)= w_{d-1}^+(|\xi|)e^{i|\xi|}+w_{d-1}^{-}(|\xi|)e^{-i|\xi|},
        $$
where
           \begin{align*}
            |w_{d-1}^{\pm}(r)|\lesssim \langle r\rangle^{-\frac{d-1}{2}},\qquad \mbox{ and }\qquad |\partial_\theta w_{d-1}^{\pm}(r)|\lesssim \langle r\rangle^{-\frac{d+1}{2}}.
        \end{align*}
        Using these and \eqref{stirlapp} on the right hand of \eqref{Equation: Better Jacobi Asymptotics}, we obtain the representation given in 
        \eqref{eq:Pn} and \eqref{Equation: Bound for Zn reduction}.  The bound for $Y_n(\theta)$ follows from the above by noting that 
        \begin{multline}\label{Equation: Zonal Harmonic}
            Y_n(\theta) := \sqrt{\tfrac{(2n+d-1)\Gamma(n+d-1)\Gamma(n+1)}{2^{d-1}\Gamma(n+1+\frac{d-2}{2})^2}}P_n^{\frac{d-2}{2},\frac{d-2}{2}}(\cos(\theta))\\
            =  n^{1/2}\left(1+O\left(\tfrac{1}{n}\right)\right)P_n^{\frac{d-2}{2},\frac{d-2}{2}}(\cos(\theta)).
        \end{multline}

        The final claim of the lemma follows from the prior work and an expansion of \cite[Chapter 12, Section 13]{OlF}:
\[
Y^{n,k}(\theta,\phi) = (-1)^ke^{ik\phi}n^{1/2}\big(\frac{\theta}{\sin\theta})^{1/2}\big( J_{|k|}((n+1/2)\theta)+O_{|k|}(\frac{1}{n}) \operatorname{env}J_{|k|}((n+1/2)\theta)\big),
\]
where $\operatorname{env}$ denotes the envelope of $J_{|k|}$ and $\theta\in(0,\pi/2)$. The asymptotics then follow exactly as above, where the weaker error bound is due to the error only being in terms of $J_{|k|}$.
\end{proof}
 
Our next aim is to introduce Besov spaces by utilizing the following proposition and to obtain related dimension bounds for the graphs of continuous functions on the spheres.  
\begin{proposition}[\cite{MR2253732, MR2237162}]
Let $a\in C^\infty([0,\infty))$ satisfy supp $a\subset[1/2,2]$, $|a(t)| > c>0$ for $t\in [3/5, 5/3]$, and $a(t)+a(2t)=1$ if $t\in[1/2,1]$, and define
\[
\Phi_0(t):= Z_0(t), \quad\Phi_j(t): = \sum_{n=0}^\infty a(2^{-j+1}n)Z_n(t), \quad t\in[-1,1],\,\,\,j\geq 1,\,\,\,\text{ and}
\] 
\[
    \mathcal{K}_j(\cos\theta) := \sum_{i=0}^{j}\Phi_i(\cos\theta).
\]
Then for all $j,k\geq 0$
\begin{equation}\label{Equation: Kernel Bound}
 |\Phi_j(\cos\theta)| + |\mathcal{K}_j(\cos\theta)|\lesssim_{d,k} \frac{2^{jd}}{(1+2^j\theta)^k},
\end{equation}
and if $f\in L^p$ for $1\leq p<\infty$  and $f\in C^0$  when $p=\infty$, then
\[
   f*\mathcal{K}_j\overset{j\to\infty}{\to} f \mbox{ in }L^p.
\]
\end{proposition}

We define the smooth cut-off Besov spaces, \cite{MR2253732}, as:
\[
\|f\|_{B^{\alpha}_{p, \infty}} := \sup_{j \geq 0} 2^{j\alpha}  \|\mathbb P_{2^j} f\|_{L^p } := \sup_{j \geq 0} 2^{j\alpha}  \|f* \Phi_j \|_{L^p }.
\]
With these preliminaries out of the way, using \eqref{Equation: Kernel Bound} it is then standard exercise in dyadic decomposition to find the following characterization of $C^\alpha(\mathbb{S}^d)$:  
\begin{theorem}\label{cesaroholder}
For $0 < \alpha < 1$, if $f \in C^{\alpha}(\mathbb{S}^{d})$ then $\|f*\mathcal{K}_j - f\|_{\infty} \lesssim 2^{-j\alpha}$. Conversely, $\|f*\mathcal{K}_j - f\|_{\infty} \lesssim 2^{-j\alpha}$ implies $f \in C^{\alpha}(\mathbb{S}^{d})$. 
\end{theorem}
This theorem leads to  
\begin{theorem}\label{upperS}
The spaces $C^{\alpha}(\mathbb{S}^{d})$ and $B^{\alpha}_{\infty, \infty}(\mathbb{S}^{d})$ coincide for $\alpha \in (0,1)$.  In particular, if $f \in B^\alpha_{\infty, \infty}(\mathbb{S}^{d})$, the fractal dimension of its graph must be bounded above by $(d+1) - \alpha$.  
\end{theorem}

For proofs of these theorems using C\'{e}saro means we refer the reader to \cite{huynh2022study}-- the proof of the above follows in the same manner, utilizing \eqref{Equation: Kernel Bound}.  
In order to obtain a version of  the Deliu-Jawerth theorem for $\mathbb{S}^d$ we will need to introduce the spherical analogue of translation.

\begin{definition}\cite[Prop 2.1.5]{DX}\label{shiftop}
Let $\theta \in [0,\pi]$ and $f \in L^2(\mathbb{S}^{d})$. For $x \in \mathbb{S}^{d}$, let $\mathbb{S}^{\perp}_{x} := \{y \in \mathbb{S}^{d}: \langle x,y \rangle = 0\}$ be the equator in $\mathbb{S}^{d}$ with respect to $x$. The average shift operator $T_{\theta}$ is defined as
\[T_{\theta}f(x) := \frac{1}{\omega_{d-1}} \int_{\mathbb{S}^{\perp}_{x}}  f(x\cos\theta + u\sin\theta) d\sigma(u).\]
For  $g: [-1,1] \to \mathbb{R}$, we have 
\[f * g(x) = \frac{\omega_{d-1}}{\omega_{d }} \int_{0}^{\pi} T_{\theta}f(x) g(\cos\theta)  (\sin\theta)^{d-1} \, d\theta.\]
\end{definition}
$T_{\theta}f(x)$, as defined, is really an average of the values of $f$ evaluated at all points $y$ such that the geodesic distance $d(x,y) = \theta$. Consequently, it is also referred to in literature as the average shift operator on $\mathbb{S}^{d}$. It has all the properties we desire from a translation operator; in addition to the convolution identity above, we have 
\begin{lemma}\cite[Pg. 32]{DX}
For $f \in L^p(\mathbb{S}^{d}), 1 \leq p < \infty$ or $f \in C^{0}(\mathbb{S}^{d})$ for $p =\infty$,
\[\|T_{\theta}f\|_{L^p} \leq \|f\|_{L^p}, \qquad \qquad \lim_{\theta \to 0+} \|T_{\theta} f - f\|_{L^p} = 0.\]
\end{lemma} 
Moreover, $T_{\theta}$  respects the symmetries of $\mathbb{S}^{d}$ by definition.  
We also need the following counting lemma. 
\begin{lemma}\label{boxcount}
Let $\theta \in [0,\pi]$. For $f: \mathbb{S}^{d} \to \mathbb{R}$ continuous function, 
\[ \|T_{\theta}f - f\|_{L^1(\mathbb{S}^{d})} \lesssim \theta^{d+1} \mathcal{N}(E, \theta)\]
where $E$ is the graph of $f$, and $\mathcal{N}(E, \theta)$ is the minimum number of balls of radius $\theta > 0$ necessary to cover $E$.
\end{lemma}
\begin{proof}
Let us denote the cap of radius $\theta$ centered at $x \in \mathbb{S}^{d}$ by $\Theta = \Theta(x)$. If $y\in \Theta$ then we may bound 
\[
|f(x)-f(y)|\lesssim \theta N_{\theta}^x,
\]
where $N_{\theta}^x$ is the minimum number of balls of radius $\theta$ required to cover the graph of $f$ above $\Theta$. This is independent of $y$, so this must also hold for the average over $y\in\partial\Theta$, and hence
\[
|f(x)-T_\theta f(x)|\lesssim \theta N_{\theta}^x.
\]

We now decompose $\mathbb{S}^{d}$ into a finite number of $\Theta(x_i)$ caps, each centered at $x_i$. Under this decomposition we find
\begin{multline*}
    \|f-T_\theta f\|_{L^1} \lesssim \int \theta N_\theta^x\,d\sigma(x)  
      \lesssim \sum_{\Theta(x_i)}\theta \int_{\Theta(x_i)}N_\theta^x\,d\sigma(x) \\
     \lesssim \sum_{\Theta(x_i)}\theta \int_{\Theta(x_i)}N_{5\theta}^{x_i}\,d\sigma(x) 
    \lesssim \theta^{d+1} \sum_{\Theta(x_i)}N_{\theta}^{x_i}\lesssim \theta^{d+1} N(E,\theta),
\end{multline*}
as desired.
\end{proof}

We are now ready to prove a key result that gives lower bounds for graph dimension. The original theorem was proven by Deliu and Jawerth in  \cite{DeJa} for continuous functions $\mathbb{T} \to \mathbb{R}$. The proof we present below extends the ideas used in the  proof given in \cite[Theorem 2.24]{BurakNikos} to continuous functions $\mathbb{S}^{d} \to \mathbb{R}$, after translating the integral over $\mathbb{S}^{d}$ into the integral involving $T_\theta$ over the interval $[0,\pi]$.

\begin{theorem}\label{Spherical Deliu Jawerth}
For a continuous function $f: \mathbb{S}^{d} \to \mathbb{R}$, the graph $E$ of $f$  has fractal dimension $D \geq (d+1) - s$ provided that $f \notin \bigcup_{\epsilon > 0} B^{s+\epsilon}_{1, \infty}(\mathbb{S}^{d})$.
\end{theorem}
\begin{proof}
It suffices to prove if $f$ is continuous and $\sup_{j \in \mathbb{N}} 2^{js} \| f*\Phi_j\|_{L^1} = \infty$ for some $0 < s < 1$, then the graph $E$ of $f$ has dimension $D \geq (d+1)- s$. 

Recalling Definition \ref{shiftop} on the operator $T_{\theta}$, we rewrite $\|\mathbb P_{2^j} f\|_{L^1}$ (for $j\geq 1$):
\begin{align*} 
\|\mathbb P_{2^j}f\|_{L^1_x} &\ \sim \left\|\int_{\mathbb{S}^{d}} \Phi_j(\langle x,y\rangle) [f(y) - f(x)] \, d\sigma(y)\right\|_{L^1_x} \\
&\ \sim \left\|\int_{0}^{\pi} \Phi_j(\cos\theta) [T_{\theta}f(x) - f(x)](\sin(\theta))^{d-1} \, d\theta\right\|_{L^1_x}.
\end{align*}
Let us consider the inner integral, and recall \eqref{Equation: Kernel Bound}. Dyadically splitting $\theta$ and choosing $k>d+s$, we find by Lemma \ref{boxcount}:
\begin{align*}
\|\mathbb P_{2^j} f\|_{L^1_x}&\lesssim\int_0^\pi |\Phi_j(\cos\theta)|\|T_\theta f-f\|_{L^1_x} \theta^{d-1}\,d\theta\\
&\lesssim \int_0^\pi 2^{dj}(1+2^j\theta)^{-k}\theta^{2d}\mathcal{N}(E, \theta)\,d\theta\\
&\lesssim \sum_{\ell=0}^\infty \int_{2^{-\ell-1}<\theta\leq 2^{-\ell}}2^{dj-2d\ell}(1+2^{j-\ell })^{-k}\mathcal{N}(E,2^{-\ell})\,d\theta\\
&\lesssim \sum_{\ell=0}^\infty 2^{dj-2d\ell-\ell}(1+2^{j-\ell })^{-k}\mathcal{N}(E,2^{-\ell}).
\end{align*}
Multiplying by $2^{sj}$ and taking the supremum in $j$, we find
\[
\|f\|_{B^s_{1,\infty}}\lesssim \sum_{\ell=0}^\infty \sup_{j\geq 0}2^{dj-2d\ell-\ell+sj}(1+2^{j-\ell })^{-k}\mathcal{N}(E,2^{-\ell})\lesssim \sum_{\ell=0}^\infty 2^{-\ell(d+1-s)}\mathcal{N}(E,2^{-\ell}),
\]
as the supremum is attained for $j = \ell $ (since $k>d+s$).  As this must diverge, we conclude that $\mathcal{N}(E,2^{-\ell}) \gtrsim 2^{\ell(d+1-s)}/\ell^2$ infinitely often, and hence the claim follows.
\end{proof}

 \section{Talbot Effect On the Sphere}\label{sec:TalbotSphere}
 In this section we will prove theorems displaying fractal behavior of solutions to the linear Schr\"odinger equation \eqref{Equation: Linear Schrodinger}. The propagator  of the Schr\"odinger equation \eqref{Equation: Linear Schrodinger}  on $\mathbb S^d$
is given by 
\be\label{eq:SchProp}
e^{it\bigtriangleup }f = \sum_{n}e^{itn(n+d-1)}\proj_{n}f,
\ee
where $\proj_n f$ is defined in \eqref{Equation: Spherical Reproducing Kernel} as the projection to eigenspace corresponding to the eigenvalue
$n(n+d-1)$. 

\begin{remark}\label{rmk:sharpLP}
For dyadic $N$,  we define the sharp cut-off Littlewood-Paley projection operators $P_N$ by
\begin{align*}
P_N(f) = \sum_{N\leq |n| < 2N} \proj_{n}f=f* \Big(\sum_{N\leq |n| < 2N} Z_n\Big).
\end{align*} 
It suffices to obtain upper bounds for these projections as the projections with smooth cut-offs  are uniformly bounded in $L^p$ spaces. 
\end{remark}

   In addition to the results we presented in the previous section,  we will make use of $L^4(\mathbb{S}^d\times[0,2\pi])$ estimates of \cite{BurqStrichartz, BurqMultilinearStrichartz}:
   \begin{align}\label{Equation: Sphere Strichartz}
    \|e^{it\bigtriangleup  } f(x,t)\|_{L^4_{x,t}(\mathbb{S}^d\times[0,2\pi])}\lesssim_\varepsilon 
    \begin{cases}
    \|f\|_{H^{\frac{1}{8}+\varepsilon} (\mathbb{S}^2)}   & d = 2\\
    \|f\|_{H^{\frac{d-2}{4}+\varepsilon} (\mathbb{S}^d)}   & d\geq 3,
    \end{cases}
\end{align}
where 
$$\|f\|_{H^s(\mathbb{S}^d)} := \sqrt{\sum_{N \text{ dyadic} } N^{2s } \|P_Nf\|_{L^2 (\mathbb{S}^d)}^2}.
$$
Up to $\varepsilon$, the above estimates are  optimal $L^4$ bounds, see \cite{BurqMultilinearStrichartz} and Remark \ref{Remark: Appendix Saturation}.
It's worth noting that what is known on $\mathbb{S}^d$ is much weaker than on $\mathbb{T}^d$. Specifically, \cite{Bourgainl2} obtained the full gamut of Strichartz estimates for $\mathbb{T}^d$:
\begin{align}\label{Equation: Torus Strichartz}
    \|e^{it\bigtriangleup_{\mathbb{T}^d}} f(x,t))\|_{L^{p}_{x,t}(\mathbb{T}^d\times[0,2\pi])}\lesssim_\varepsilon 
    \begin{cases}
  \| f\|_{H^{\frac{d}{2}-\frac{d+2}{p}+\varepsilon}(\mathbb{T}^d)} & p\geq \frac{2(d+2)}{d}\\
    \| f\|_{H^\varepsilon (\mathbb{T}^d)} & 2 < p < \frac{2(d+2)}{d}.
    \end{cases}
\end{align}
Unlike \eqref{Equation: Sphere Strichartz}, these estimates always include a region of $p>2$ for which there is only an  $\varepsilon$-derivative loss.
In Lemma~\ref{Lemma: Improved Bilinear Strichartz} below, we obtain improved $L^4$ bound with only $\varepsilon$ loss when restricted to zonal spherical harmonics.  

 The next lemma is standard. Specifically, it demonstrates the expected square root cancellation for weighted Weyl sums with decaying weights, see \cite{BurakNikos,ErdShak,huynh2022study}.

\begin{lemma}\label{Lemma: Exponential Sum Bound}
Consider the exponential sum $\sum_{n = N}^{u} e^{in^2t+inx} a^n b_n$ for $x \in \mathbb{R}$. Assume $a \in [0,1]$ and for each $n \in \mathbb{N}$, $b_n$ satisfies 
\[|b_n| \les \frac{1}{n^p}, \hspace{1cm} |b_n - b_{n-1}| \les \frac{1}{n^{p+1}},\]
for some $p \in \mathbb{R}$. 
It follows that for almost every $t$, we have for all $N \in \mathbb{N}$:
\begin{align}
\sup_{N \leq u \leq 2N} \sup_{x \in \mathbb{R}} \bigg| \sum_{n = N}^{u} e^{in^2t+inx} \, a^n \, b_n\bigg| \lesssim_{t} N^{1/2-p}.
\end{align}
\end{lemma}
\begin{remark}
The above factor of $a^n$ is not important for many of our applications, but is included because it makes the presentation of applications to functions supported on Gaussian beams simpler. 
\end{remark}

With this out of the way, we're ready to prove our first theorem.

\begin{proof}[Proof of Theorem \ref{Theorem: General Sphere Theorem}]
By Theorem~\ref{upperS}, it suffices to prove that  for almost all $t$, $u(t,\cdot)$ belongs to $B^{\gamma-}_{\infty,\infty}$. By Remark~\ref{rmk:sharpLP}, it suffices to prove that 
\[\|  P_N(u(t,x))\|_{L_{x}^{\infty}(\mathbb{S}^{d})} \lesssim N^{-\gamma+}.\]
We write $P_N(u(x,t)) = f_N * H_N(x,t)$, where 
\be\label{eq:HN}
    H_N (\cos \theta,t) := \sum_{N < n \leq 2N} e^{itn(n+d-1)}Z_n(\cos\theta)
\ee
and the convolution is the spherical convolution.
Using \eqref{SphericalYoung}, we have  (for $\frac{1}{p} + \frac{1}{q} = 1$):
\begin{align}\label{Equation: Youngs Reduction General Sd}
\|P_N(u(t,x))\|_{L_{x}^{\infty}(\mathbb{S}^{d})}\lesssim \| f_N\|_{L^{p}(\mathbb{S}^d)} \left\| H_N(\cdot,t)\right\|_{L_w^q([-1,1])},
\end{align}
where $\|\cdot\|_{L_w^q([-1,1])}$ is defined by \eqref{Equation: wlp norm}.

By Lemma~\ref{Lemma: Zonal and convolution bounds}, we have (by symmetry, $\theta \in [0, \frac{\pi}{2}]$)
\begin{align*}
    Z_n(\cos(\theta)) = b_n^+(\theta)e^{in\theta}+b_n^-(\theta)e^{-in\theta}+E(\theta,n),
\end{align*}
where $b_n^\pm$ satisfies \eqref{Equation: Bounds on Convolution Zonal} and hence also the conditions of Lemma~\ref{Lemma: Exponential Sum Bound} with $N< n\leq 2N$, $a = 1$ and  $p$ depending on $\theta, N$ is
\begin{equation}\label{Equation: p for H in convolution}
    p(\theta, N) = 
    \begin{cases}
     -\frac{d-1}{2} & \theta  \in [\frac{1}{N}, \frac\pi2]\\
     -(d-1)  & \theta  \in [0, \frac{1}{N} ].
    \end{cases}
\end{equation} 

Now, by incorporating the phase factors from above into the exponential and using the trivial estimate on the error, we see that  Lemma~\ref{Lemma: Exponential Sum Bound} yields, 
for almost all $t$:
\begin{align*}
    \left|H_N(\cos \theta, t)\right|\lesssim_{t} \frac{N^{d-\frac{1}{2}}}{\langle N\theta\rangle^{\frac{d-1}{2}}}.
\end{align*}  
Using this we have 
\begin{multline*}
 \left\| H_N(\cdot,t)\right\|_{L_w^q([-1,1])}^q \les \int_{0}^{\frac{\pi}{2}}\left|H_N(\cos \theta, t)\right|^q\sin^{d-1}\theta\,d\theta 
\\
\les_t \int_{0}^{\frac{\pi}{2}} \frac{N^{q(d-\frac{1}{2})}}{\langle N\theta\rangle^{q\frac{d-1}{2}}} \theta^{d-1} \,d\theta 
\les N^{q(d-\frac{1}{2}) -d} +  \begin{cases}
    N^{\frac{dq}{2}+} & \mbox{if }2+\frac{2}{d-1}\geq q\\
    N^{q(d-\frac{1}{2}) -d} &\mbox{if } q > 2+\frac{2}{d-1}.
    \end{cases}
\end{multline*} 
Together with \eqref{Equation: Youngs Reduction General Sd} and the hypothesis $\|f_N\|_{L^p} \lesssim N^{-\left(\frac{d}{2}+s\right)}$, we then have
\begin{align*}
    \|P_n(u(\cdot, t))\|_{L^\infty_x(\mathbb{S}^d)}\lesssim \| f_N\|_{L^p}\left\|H_N\right\|_{L_w^q([-1,1])}\lesssim\max\left(N^{-s +}, N^{d\left(\frac{1}{p}-\frac{1}{2}\right)-\frac{1}{2}-s+}\right).
\end{align*}

This then gives $u(t,x) \in C^{\gamma-}$ for almost all $t$ and $\gamma = \min\{s, s +\frac{d+1}{2}-\frac{d}{p}, 1\}$. It also implies the upper bound of $(d+1) - \gamma$ on the fractal dimension of graph of $u(t,x)$ for almost every $t$. 

For the lower bound, following the arguments of \cite{ErdShak}, we see that the assumption on $\|f_N\|_{L^p_x}$ and Sobolev embedding imply
($d=2$)
\[
\|f_N\|_{L^2_x}\lesssim N^{ \frac{2}{p}-1}\|f_N\|_{L^p_x}\lesssim N^{\frac{2}{p}-2-s},
\]
so that we find by \eqref{Equation: Sphere Strichartz}
\[
\|\langle \nabla \rangle^{-(\frac18+\frac{2}{p}-2-s)-}u\|_{L^4_{x,t}}\lesssim \|\langle \nabla \rangle^{-(\frac{2}{p}-2-s)-}f\|_{L^2_x}\lesssim 1.
\]
 Thus, for almost every time, $t$, we find
\[
\|P_N(u)\|_{L^4_{x}}\lesssim_t N^{(\frac18+\frac{2}{p}-2-s)+}.
\]
We now assume that  $f\not\in H^{s+2-\frac{2}{p}}$, so that interpolation between the $L^1$ and the $L^4$ bound gives
\[
\sup_N N^\gamma\|P_N(u)\|_{L^1_x} = \infty,
\]
for $\gamma > 2-\frac{2}{p}+\frac14+s.$ It follows by Theorem~\ref{Spherical Deliu Jawerth} that the fractal dimension of the graph of $u$ is bounded below by $\max\left(\frac{3}{4}+\frac{2}{p}-s, 2\right) $ for almost all $t$.
\end{proof}
\begin{remark}
The upper bound above is best when $p = \frac{2d}{d+1}\to 2$ as $d\to\infty$. At this level, the bound for $H_N$ begins to match the bound for the torus, see Theorem \ref{Theorem: General TOrus}.
\end{remark}
 
\subsection{Specific Expansions on $\mathbb{S}^2$}\label{Section: Spherical Specifics}
 In this subsection we explore dimension bounds for functions supported on specific spherical harmonics in the case that $d = 2$. In particular, we will focus primarily on functions supported on the Zonal harmonics and Gaussian beams, which are of interest because they are the extremal harmonics in some sense (explored further in Appendix \ref{Appendix: Strichartz}). However, we also demonstrate results for more general expansions (of the same form as Zonal and Guassian Beams) that follow from the same methods.
 
 When $d = 2$, we have an explicit formula for the spherical harmonics of degree $n$ (see for example  \cite[Section 1.6]{DX}):
\begin{theorem}\label{Theorem: SPherical Harmonic Specific Form}
On $\mathbb{S}^2$, let $\theta, \phi$ denote the azimuthal and polar angles respectively in spherical coordinates ( $0\leq \theta \leq \pi, 0 \leq \phi < 2\pi$). For $-n \leq k \leq n$, Define 
\begin{align}\label{2dspherical}
Y_n^k (\theta, \phi) := \sqrt{\tfrac{(2n+1)(n-k)!}{(n+k)!}} P_{n}^k(\cos \theta) e^{ik\phi}.
\end{align}
Then $\{Y_n^k: n \in \mathbb{N}_0, -n \leq k \leq n\}$ forms an orthonormal basis of $L^2(\mathbb{S}^2)$.
\end{theorem} 
By inductive construction, one can also obtain basis elements for $L^2(\mathbb{S}^{d})$ for $d >2$, see \cite{AH}.

The Zonal harmonic of degree n, denoted\footnote{We comment that when $d>2$ these are denoted $Y_n$. When $d=2$ the spherical harmonics are usually denoted as some variant of $Y_n^k$ for $|k|\leq n$. The case that $k=0$ corresponds to the Zonal harmonics on $\mathbb{S}^2$.} $Y_n^0$, has an explicit form given by \eqref{Equation: Zonal Harmonic}. In particular, we find
\begin{equation}
    Y_n^0(\theta, \phi) = \sqrt{2n+1}P_n(\cos \theta),
\end{equation}
where $P_n$ is now the Legendre polynomial of degree $n$. Gaussian beams, denoted $Y^{\pm n}_n$, have a similarly nice expression of the form
\begin{equation}
    Y^{\pm n}_n(\theta,\phi) = \sqrt{(2n+1){\tbinom{2n}{n}}}\frac{(\mp 1)^n}{2^n}\sin^n(\theta).
\end{equation}

\begin{proof}[Proof of Theorem \ref{Theorem: Zonal Bound}]
We will show for almost all $t$ and $N$ dyadic
\[\|P_N(u(\cdot,t))\|_{L^{\infty}} \lesssim N^{-(p-1)+}.\]
We use again Lemma \ref{Lemma: Zonal and convolution bounds} and summation-by-parts as in the proof of Theorem \ref{Theorem: General Sphere Theorem}, noting that $b_n = a_n\mathfrak{b}_n^\pm$ and $a = 1$ satisfies the hypothesis of Lemma \ref{Lemma: Exponential Sum Bound} with
\begin{align*}
    |b_n| \lesssim  \frac{n^{\frac{1}{2}-p}}{\langle n\theta\rangle^{\frac{1}{2}}}
\end{align*}
It follows that
\begin{align*}
    \|P_n(u(\cdot, t))\|_{L^\infty_x}\lesssim \sup_{\theta}\frac{N^{1-p}}{\langle N\theta\rangle^{\frac{1}{2}}}\lesssim N^{1-p},
\end{align*}
establishing the upper bound claimed.

In order to establish a lower bound we must refine the Strichartz estimate \eqref{Equation: Sphere Strichartz}.
This is done in the appendix, Lemma \ref{Lemma: Improved Bilinear Strichartz}. In particular, we find 
\[
\|u\|_{L^4_{x,t}}\les \|f \|_{H^{\varepsilon}},
\]
for any $\varepsilon > 0$. We now assume that $f\not\in H^{p-\frac{1}{2}+}(\mathbb{S}^2)$. Interpolation as in Theorem \ref{Theorem: General Sphere Theorem} with the improved $L^4$ bound forces the dimension to be bounded below by
\[
\max\left(3 - \left(p-\tfrac{1}{2}\right), 2\right) = \max\left(\tfrac{7}{2} - p, 2\right).  \qedhere
\]
\end{proof}
\begin{remark}
The above lower bound is only meaningful when $1\leq p\leq \frac{3}{2}.$
\end{remark}

We can also easily extend the above upper bound to zonal harmonics on $\mathbb{S}^d$ by using the full asymptotics of Lemma~\ref{Lemma: Zonal and convolution bounds}. In particular, we find Corollary \ref{Corollary: Zonal In Sd} holds, whose proof follows in the exact same way as above.

One doesn't have to stop at just Zonal harmonics. Inded, Lemma~\ref{Lemma: Zonal and convolution bounds} yields asymptotics for $Y^k_n$, under the assumption that $k$ is fixed. Using these asymptotics we may establish dimension bounds similar to Theorem~\ref{Theorem: Zonal Bound}.
\begin{lemma}\label{Lemma: Genearal fixed k 2d lemma}
    Let $k$ be fixed and $f(\theta,\phi) = \displaystyle\sum_{n\geq k}^{\infty} a_n \, Y_n^k(\theta,\phi)$. 
    %\[u(\theta,\phi,t)= \sum_{n=0}^{\infty} e^{in(n+1)t} \, a_n  \,  Y_n(\theta,\phi)\] is the solution to the linear Schr\"{o}dinger equation on $\mathbb{S}^2$ emanating from $f$. 
    If for some $1< p < 2$ we have
    \[|a_{ n}| \lesssim \tfrac{1}{n^{p}}, \text{ and } \quad |a_{ n} - a_{ n-1}| \lesssim \tfrac{1}{n^{p+1}}\]
    for all $n \in \mathbb{N} \cup \{0\}$, then for almost all $t$,  $u(x,t) \in C^{(p-1)-}$, and hence 
    $\dim_t(f)\leq 4-p$. 
    If, in addition, $f\not\in H^{p-\frac{1}{2}+}(\mathbb{S}^2)$, then we also find $\dim_t(f)\geq \max\left(\tfrac{15}{4} - p, 2\right)$.
\end{lemma}
\begin{proof}
The upper bound follows immedietly from Lemma~\ref{Lemma: Zonal and convolution bounds}, as in Theorem~\ref{Theorem: Zonal Bound}. The lower bound similarly follows from interpolation using the Strichartz estimate \eqref{Equation: Sphere Strichartz}.
\end{proof}

Instead of fixing $k$, we also establish bounds when $k = -n$, which corresponds to the case of Guassian beams.

\begin{theorem}\label{Theorem: Gaussian Beams}
  Let $f(x) = f(\theta,\phi) = \displaystyle\sum_{n=0}^{\infty} a_n Y_n^{-n}(\theta,\phi)$.
  If for some $\frac{3}{4} < p < \frac{5}{4}$ we have, 
    \[|a_{ n}| \lesssim \frac{1}{n^{p}}, \text{ and } \quad |a_{ n} - a_{ n-1}| \lesssim \frac{1}{n^{p+1}}\]
      for all $n \in \mathbb{N} \cup \{0\}$, then %$u(x,t) \in C^{(p-3/4)-}$ 
      for almost all $t$ and $\dim_t(f)\leq \frac{15}{4}-p$.
      
      If, in addition, $f \notin H^{p-1/2+}$, then  $\dim_t(f)\geq \max(\frac{13}{4} - p,2).$
    \end{theorem}
\begin{proof}
Recall that (referencing Theorem \ref{Theorem: SPherical Harmonic Specific Form})
\begin{align*}
Y_n^{-n} (\theta, \phi) &\ = \sqrt{(2n+1)(2n)!} P_n^{-n}(\cos\theta) e^{-in\phi} \\
&\ = \sqrt{\tfrac{2n+1}{(2n)!}} (-1)^n P_n^n(\cos\theta) e^{-in\phi}
\end{align*}
Furthermore, using the closed form formula for associated Legendre polynomial \cite{LF}, we have:
\[P_n^n(\cos\theta) = (-1)^n \, 2^n \, n! \, {\tbinom{n-1/2}{n}} (1-\cos^2\theta)^{n/2} = (-1)^n\,  \frac{n!}{2^n} {\tbinom{2n}{n}} (\sin\theta)^n\]

We can now rewrite $u(x,t)$ using this above information:
\begin{align*}
u(\theta,\phi,t) &= \sum_{n=0}^{\infty} e^{in(n+1)t} \, a_n Y_n^{-n}(\theta,\phi) \\
&\ = \sum_{n=0}^{\infty} a_n \, e^{in(n+1)t} \sqrt{\tfrac{2n+1}{(2n)!}} P_n^n(\cos\theta) e^{-in\phi} \\
&\ = \sum_{n=0}^{\infty} e^{in(n+1)t} e^{-in\phi} \sin^n(\theta) \, a_n \, \sqrt{2n+1} \, \sqrt{{\tbinom{2n}{n}}} \, \frac{1}{2^n}.
\end{align*}

In order to obtain an upper bound on the fractal dimension of the graph of $u(x,t)$, we need to estimate \[\|P_N(u(\cdot,t))\|_{L^{\infty}_x} =  \sup_{(\theta,\phi) \in \mathbb{S}^2} \bigg| \sum_{N < n \leq 2N} e^{in(n+1)t} e^{-in\phi} \sin^n(\theta) \, \, a_n \sqrt{2n+1} \, \sqrt{{\tbinom{2n}{n}}} \, \frac{1}{2^n}\bigg|\]
for $N$ dyadic. 

Once again, Lemma \ref{Lemma: Exponential Sum Bound} can be used for fixed $\theta$ with $a = \sin\theta$ and $b_n = a_n \sqrt{2n+1} \, \sqrt{{\tbinom{2n}{n}}} \, \frac{1}{2^n}$. We note that by Stirling's approximation \eqref{stir},
we find
\[\sqrt{{\tbinom{2n}{n}}}  = \sqrt{\tfrac{2^{2n}}{\sqrt{n}}}\  \big(c+O(\tfrac1n)\big).\]
Consequently 
\begin{align*}
|b_n| \lesssim \bigg|a_n \, \sqrt{2n+1} \, \sqrt{\tfrac{2^{2n}}{\sqrt{n}}} \, \frac{1}{2^n}\bigg|\lesssim |a_n n^{1/4}|.
\end{align*}
By the assumption on $a_n$, we obtain for $b_n$:
\begin{align}\label{Stirlingsimp}
|b_n| \lesssim \frac{1}{n^{p-1/4}};  \hspace{1cm} |b_n - b_{n-1}| \lesssim \frac{1}{n^{p+3/4}}.
\end{align}

This leads to the following for almost all $t$:
\begin{align*}
\sup_{\substack{\theta \in [0,\pi/2]\\ \phi \in [0, 2\pi]}} \bigg|\sum_{N < n \leq 2N} e^{in(n+1)t} &e^{-in\phi} \sin^n(\theta) \, \, a_n \, \sqrt{2n+1} \, \sqrt{{\tbinom{2n}{n}}} \, \frac{1}{2^n}\bigg|\lesssim \frac{N^{1/2+}}{N^{p-1/4}}= N^{-(p-3/4)+}.
\end{align*}

Given the assumption on $p$, this implies the solution $u(x,t)$ is in $C^{(p-3/4)-}$ for almost all $t$, thus the fractal dimension of its graph is bounded above by $3 - \big(p -\frac{3}{4}\big) = \frac{15}{4}-p$.

As there is no hope of improving the Strichartz estimate for these harmonics, the lower bound follows by the same interpolation argument as Theorem \ref{Theorem: General Sphere Theorem} and Lemma~\ref{Lemma: Genearal fixed k 2d lemma}.
\end{proof}

\begin{remark}
Similar statements are available for combinations of specific harmonics, but depend greatly on the specific form of the eigenfunction. Because of this it seems difficult to obtain a result akin to Theorem \ref{Theorem: Td L2} for $\mathbb{S}^2$.
\end{remark}

\begin{remark}
The above corollary isn't unique to $Y^{\pm n}_n$-- a similar statement will hold in the exact same way for any function supported on $\{Y^{g(n)}_n\}_n$ with $\left|\frac{g(n)}{n}\right|\to 1$ as $n\to\infty$.
\end{remark}

We can say more than the above theorem when off of the equator. Specifically, because of the factor of $\sin^n(\theta)$, we find that all functions supported on the Gaussian beams with polynomially growing coefficients is $C^\infty_{x,t}(S_\delta)$ for
\[
S_\delta = \left\{(\theta,\phi)\in\mathbb{S}^2\,:\, \theta\not\in\left(\tfrac{\pi}{2}-\delta, \tfrac{\pi}{2}+\delta\right)\right\},
\]
and $0 < \delta < \frac{\pi}{2}.$ This is recorded in the next corollary.

\begin{corollary}
Let $\alpha,\beta \in\mathbb{R}$, $0 < \delta < \frac{\pi}{2}$, and $f(x) = \displaystyle\sum_{n=0}^{\infty} a_n Y_n^{-n}(x)$ for some scalars $a_{n}$ with $|a_n|\lesssim n^\alpha$. If $\mu:\mathbb{R}\mapsto\mathbb{R}$ satisfies $|\mu(x)|\lesssim |x|^\beta$ and
\[u(x,t) = e^{it\mu\left(-\bigtriangleup\right)}f= \sum_{n=0}^{\infty} e^{it\mu\left(n(n+1)\right)} a_n Y_n^{-n}(x),\]
then $u\in C^\infty_{x,t}(S_\delta).$ It follows that the dimension of $u(S_\delta, t)$ is exactly $2$ for all $t$.
\end{corollary}
\begin{proof}
This is fairly trivial. Let $x = (\theta,\phi)$ and $k, \ell \in\mathbb{N}$. We then find
\begin{align*}
    |\partial_t^k\bigtriangleup^\ell u(x,t)| &=  \left|\sum_{n=0}^{\infty} e^{it\mu\left(n(n+1)\right)} \mu(n(n+1))^kn^\ell(n+1)^\ell a_n \cdot Y_n^{-n}(x)\right|\\
    &\lesssim \sum_{n=0}^\infty |n|^{2k+2\beta+\alpha+\frac{1}{4}}\sin^n\left(\frac{\pi}{2} - \delta\right)  < \infty,
\end{align*}
uniformly for $x\in S_\delta$.
\end{proof}

\section{Nonlinear Smoothing for the Zonal Cubic NLS on $\mathbb{S}^d$}\label{Section: Smoothing}

In this section we derive a smoothing statement for the cubic NLS on $\mathbb{S}^d$ restricted to 
\[
\mathcal{Z}^s(\mathbb{S}^d) := \Big\{f\in H^s(\mathbb{S}^d):\, f = \sum_{n=0}^\infty a_n Y_n \Big\}.
\]
Here $a_n=\widehat f(n)=\mF (f)(n) =\frac1{\omega_d}  \int_{S^d} f(x) Y_n(x) \ d\sigma(x)$. 

We have Parseval's identity:
$$
\sum_{n=0}^\infty \widehat f(n)\bar{\widehat g}(n)  =\frac1{\omega_d}  \int_{\mathbb S^d} f(x)\bar g(x) d\sigma(x).
$$

We also define
\be\label{kappa}
\kappa(n, n_1,...,n_j):=\mF\big(\prod_{i=1}^{j } Y_i\big)(n)= \frac1{\omega_d}  \int_{\mathbb S^d} Y_n(x) \prod_{i=1}^{j } Y_i(x) \ d\sigma(x).
\ee
Note that $\kappa$ is independent of the order of the indices. 
By \cite{GasperPositivity}, $\kappa\geq 0$ for any choice of indices, and
$\kappa=0$ if any index is strictly greater than the sum of the others.   Finally, by the Parseval's identity above, we have
$$
\sum_n \kappa(n,n_1,...,n_j)\kappa(n,m_1,...,m_\ell) =\kappa(n_1,...,n_j ,m_1,...,m_\ell).
$$

\begin{remark}
These facts about $\kappa$ allows one to perform multilinear estimates in the standard way-- that is, by assuming positivity of either the Fourier transforms or the space-time Fourier transforms and pulling absolute values in. The non-negativity specifically guarantees access to Parseval and Plancherel after pulling in the absolute values. 
\end{remark}

We consider solutions to 
\begin{equation}\label{Equation: Zonal Cubic NLS}
    \begin{cases}
    i\partial_t u+\bigtriangleup_{\mathbb{S}^d}u \pm |u|^2u = 0\\
    u(x,0) = u_0(x)\in\mathcal{Z}^s(\mathbb{S}^d).
    \end{cases}
\end{equation}
As a consequence of Lemma \ref{Lemma: Improved Bilinear Strichartz} and \eqref{Equation: Sphere Strichartz} we find the following proposition.
\begin{proposition}
For $s > \frac{d-2}{2}$ the equation   $\eqref{Equation: Zonal Cubic NLS}$ is locally well-posed with a time of existence  $T= T(\|u_0\|_{H^s})>0.$
\end{proposition}

Before moving on to the statement of the result, we first describe the nature of the result, leaving the historical background of the result to  \cite{BurakNikos,MCCONNELL2022353}. The linear group of the NLS is an isometry on $L^2$ based spaces, and hence we can not expect the linear group to present any smoothing behaviour, i.e., it cannot be in an higher index $L^2$-based Sobolev space for any $t$. 
By the Duhamel representation, we have 
\[
u(t,x)-e^{it\bigtriangleup_{\mathbb{S}^d}}u_0(x) = \mp \int_0^t e^{i(t-s)\bigtriangleup_{\mathbb{S}^d}}|u|^2(s,x)u(s,x)\,ds.
\]
One could expect such a statement to possibly hold for the nonlinear part of the evolution, i.e., the right hand side of the formula above. 
 On many periodic domains this, too, fails.
This  is best seen through the result of \cite{ErTz2}, which demonstrates that on $\mathbb{T}$ one has  a $u\|u_0\|_{L^2_x}^2$ term sitting in the non-linearity arising from resonances. This automatically precludes extra smoothness of the integral term in the Duhamel representation. The fix to this is to introduce a phase rotation in order to remove this term from the differential equation, modifying the equation to its \textit{Wick reordering}. Indeed, the correct statement is then
\[
u-e^{it\left(\bigtriangleup_\mathbb{T}+\frac1\pi \|u_0\|_{L^2_x(\mathbb{T})}\right)}u_0\in C^0\left([0, T), H^{s+\varepsilon}_x(\mathbb{T})\right),
\]
for $0\leq \varepsilon < \min(2s, 1/2),$ \cite{ErTz2}, which brings us back to the statement of Theorem~\ref{Theorem: nonlinear smoothing}.

\begin{proof}[Proof of \ref{Theorem: nonlinear smoothing}]
The proof of Theorem \ref{Theorem: nonlinear smoothing} follows from an application of Lemmas \ref{Lemma: Main Smoothing Lemma}, \ref{Lemma: Single Resonant Smoothing Lemma}, and \ref{Lemma: Double Resonant Smoothing Lemma} below to the Duhamel representation associated to \eqref{preduhamelR}, together with the local well-posedness bound for $0 \leq t < T$. 
\end{proof}
\begin{remark}
In particular, $\gamma$ is a real function depending on the solution $u$, so that the above theorem states that the solution, up to a phase rotation of the initial data, is in a smoother space than the initial data.
\end{remark}
As an application, we prove dimension bounds for the nonlinear evolution.

\begin{proof}[Proof of \ref{Theorem: Cubic NLS dimension bound}]
We first write 
\[
e^{-\mp i\int_0^t \gamma(s;\, u)\,ds}u = e^{it\bigtriangleup_{\mathbb{S}^d}}f + v,
\]
where by Theorem \ref{Theorem: Zonal Bound} and Corollary \ref{Corollary: Zonal In Sd}, we find that
\[
e^{it\bigtriangleup_{\mathbb{S}^d}}f\in C^{p-\frac{d}{2}-}.
\]
Moreover, $v\in \mathcal{Z}^{p-1/2+\varepsilon-}$ for some $\varepsilon \geq 1/2$ when $p>(d+1)/2$. It follows that $v\in C^{p-\frac{d}{2}-}$. Combining these two facts, we see that $u\in C^{p-\frac{d}{2}-}$, and hence 
\[
\dim_t(u)\leq (d+1)-\left(p-\tfrac{d}{2}\right). \,\,\,\,\,\, \qedhere
\]
\end{proof}

Before proceeding to the proof of Theorem \ref{Theorem: nonlinear smoothing} we first derive a simple result that will guide our analysis. The restriction on the indices are due to the fact that $\kappa(n_1, n_2, n_3, n)=0$ if any of the indices is strictly greater than the sum of the others. 
This restriction replaces  the relation $n=n_1-n_2+n_3$ that one sees on the torus or real line.

\begin{lemma}\label{Lemma: Symbol decomposition}
    Let $n_1, n_2, n_3, n\in \mathbb{N}\cup\{0\}$, $ n_1 \geq  n_3 $, 
    \[
    \max( n_1-n_2-n_3, n_2-n_1-n_3 , 0)\leq n\leq n_1+n_2+n_3,
    \]
    and define\footnote{$H(n_1, n_2, n_3, n)$ will often be abbreviated as $H_n$.} 
    \[
    H(n_1, n_2, n_3, n) := n(n+d-1)-n_1(n_1+d-1)+n_2(n_2+d-1)-n_3(n_3+d-1).
    \]
    Then at least one of the following must hold.
    \begin{enumerate}
        \item $n=n_1$,
        \item $\la n_1\ra \la n_2\ra \la n_3\ra \gtrsim n^{3/2}$,
        \item $|H_n|\gtrsim \max(n_1, n_2)|n-n_1|$.
    \end{enumerate}
    \end{lemma}
    
    \begin{proof}
    The proof is straightforward. We assume that all of the above are false. We first suppose that $n_2\gg n_1 $, so that the negation of the final item is immediately violated, as $H_n\gtrsim n_2^2$. We now assume $n_1\gtrsim n_2$, which also implies $n_1\gtrsim n$. Therefore the  
    negation of item 2 implies that $n_2,n_3 \leq  \la n_2\ra \la n_3\ra \ll n^{1/2} \les n_1^{1/2}$ as well as $n_1\les n$. Combining these facts we see that 
    \[
    |H_n| =| (n+n_1+d-1)(n-n_1)+n_2(n_2+d-1)-n_3(n_3+d-1)|\sim  n_1|n-n_1|,
    \]
    as the second two summands are, in magnitude, $\ll n$. This again contradicts the negation of the final item, completing the proof. 
\end{proof} 

The second case in Lemma \ref{Lemma: Symbol decomposition} corresponds to resonances, which must be handled  separately. In order to do that we need the following lemmas.

\begin{lemma}[{Szeg\"o, \cite[Theorem 8.21.13]{szego}}]\label{Lemma: Generic Bounds for Jacobi Polynomails}
    Let $0 < c < \pi$ be fixed. Then (uniformly) for $\theta\in[\frac{c}{n}, \pi-\frac{c}{n}],$ we have
    \[
    P_n^{(\frac{d-2}{2},\frac{d-2}{2})}(\cos\theta) = n^{-\frac{1}{2}}k(\theta)\left(\cos(M\theta+\gamma)+\tfrac{O(1)}{n\sin\theta}\right),
    \]
    where $M_n =n+\frac{d-1}{2}$,  $\gamma  = -\frac{d-1}{2}\cdot\frac{\pi}{2}$, and 
    $$
        k(\theta)  = 2^\frac{d-2}{2}\pi^{-\frac{1}{2}}\sin(\theta)^{-\frac{d-1}{2}}.
    $$
    In the remaining region we find that $|P_n^{(\frac{d-2}{2},\frac{d-2}{2})}(\cos\theta)| \sim n^\frac{d-2}{2}.$
\end{lemma}
\begin{lemma}\label{Lemma: Resonance Bound}
Let $0< n_1, n_2 \leq  n$ and $d\geq 2$. Then
\[
    \frac{1}{\omega_d}\int_{\mathbb{S}^d}Y_{n_1} Y_{n_2} Y_n^2\,d\sigma= \frac{1}{\pi\omega_d}\int_0^\pi Y_{n_1}(\theta)Y_{n_2}(\theta)\,d\theta + O\Big(\tfrac{(n_1n_2)^{\frac{d-1}{2}+}}{n}\Big).
\]
Moreover, for $n_1\geq n_2\geq n_3\geq n_4$ we have the estimate
\[
\frac{1}{\omega_d}\int_{\mathbb{S}^d}Y_{n_1}Y_{n_2}Y_{n_3}Y_{n_4}\,d\sigma = O\left((n_3n_4)^{\frac{d-2}{2}+}\right).
\]
%where the logarithmic loss can be taken to only be in dimension $d=3$.
\end{lemma}
\begin{remark}
Before moving on to the proof, we remark that the integrals on two sides are with respect to different measures; the one on the right hand side lacks the factor of $\sin^{d-1}\theta$ that arises due to the measure $d\sigma$. % It follows that this term can be rather big.
\end{remark}
\begin{proof}
We begin first with a calculation in an attempt to understand the action of $(Y_n(\theta))^2\sin^{d-1}(\theta)$. Let
\begin{align}\label{Equation: a asymptotic}
\alpha(n,d)  = \frac{(2n+d-1)\Gamma(n+d-1)\Gamma(n+1)}{\Gamma(n+\frac{d}{2})^2} 
 = (2n+d-1)\left(1+O\left(\tfrac{1}{n}\right)\right)
\end{align}
by Stirling's approximation \eqref{stir}. Then, for $\theta\in\left[\frac{1}{n}, \pi-\frac{1}{n}\right]$, Lemma \ref{Lemma: Generic Bounds for Jacobi Polynomails} and \eqref{Equation: a asymptotic} give 
\begin{align*}
    Y_n(\theta)^2\sin^{d-1}(\theta) &= \frac{\alpha(n,d)}{n\pi}\left(\cos\left(\left(n+\tfrac{d-1}{2}\right)\theta - \tfrac{(d-1)\pi}{4}\right) + \tfrac{O(1)}{n\sin\theta}\right)^2\\ 
    &= \frac{2}{\pi}\cos\left(\left(n+\tfrac{d-1}{2}\right)\theta - \tfrac{(d-1)\pi}{4}\right)^2 + O\left(\tfrac{1}{n\sin(\theta)}\right).
\end{align*}
A calculation for the cosine  term above shows that it has a very strong localization near its mean:
\begin{equation}
    \int_a^b\frac2\pi \cos\left(\left(n+\tfrac{d-1}{2}\right)\theta - \tfrac{(d-1)\pi}{4}\right)^2\,d\theta = \tfrac{1}{\pi}(b-a)+O\left(\tfrac{1}{d+n}\right).
\end{equation} We then let 
\begin{equation}
    \omega(\theta; n,d) = Y_n(\theta)^2\sin^{d-1}(\theta) - \tfrac{1}{\pi},
\end{equation}
and remark that for any $\frac{1}{n} < a < b <\pi-\frac{1}{n}$, $\omega$ satisfies
\begin{equation}\label{Equation: Smoothing w bound}
    \int_a^b \omega(\theta;n,d)\,d\theta = O\left(\tfrac{1}{n^{1-}}\right),
\end{equation}
where the implicit constants depend (harmlessly so) on the fixed $d$ and the $\varepsilon$ hidden in the $1-$ notation. We now note that, by symmetry, it suffices to consider  
\begin{equation}
    \int_0^{\frac{\pi}{2}} Y_{n_1}Y_{n_2}(Y_n)^2\sin(\theta)^{d-1}\,d\theta,
\end{equation}
for $0 < n_2\leq n_1\leq n$.

In what is to follow, our main tool will be Lemma \ref{Lemma: Generic Bounds for Jacobi Polynomails}. On $[0,\frac1n]$ we find that
\[
Y_{n_1}Y_{n_2}(Y_n)^2 = O\left((n_1n_2)^{\frac{d-1}{2}}n^{d-1}\right),
\]
so that 
\begin{align}
    \int_{0}^{\frac1n} Y_{n_1}Y_{n_2}(Y_n)^2\sin(\theta)^{d-1}\,d\theta  \les  (n_1n_2)^{\frac{d-1}{2}}n^{d-1}\int_0^{\frac{1}{n}}\sin(\theta)^{d-1}\,d\theta 
     \les  \frac{(n_1n_2)^{\frac{d-1}{2}}}{n},\label{Equation: Smoothing Lemma A1 bound}
\end{align}
where the same inequality also holds for $\int_0^{1/n} Y_{n_1}Y_{n_2} d\theta$. 

On $[\frac1n,\frac\pi2]$, we must remove the mean. Specifically,  we rewrite
\[
\int_{\frac1n}^{\frac\pi2} Y_{n_1}Y_{n_2}(Y_n)^2\sin(\theta)^{d-1}\,d\theta = \int_{\frac1n}^{\frac\pi2} Y_{n_1}Y_{n_2}\omega(\theta;n,d)\,d\theta + \frac{1}{\pi}\int_{\frac1n}^{\frac\pi2} Y_{n_1}Y_{n_2}\,d\theta.
\]
Therefore, it remains  to prove that
$$
\Big|\int_{\frac1n}^{\frac\pi2} Y_{n_1}Y_{n_2}\omega(\theta;n,d)\,d\theta\Big| \les  \frac{(n_1n_2)^{\frac{d-1}{2}}}{n}.
$$
To utilize the average bound on $\omega(\theta;n,d)$, we need to apply integration by parts, for which we need the Jacobi polynomial identity \cite[Equation 4.7.14]{szego}:
\begin{equation}
    \partial_\theta P^{\frac{d-2}{2}, \frac{d-2}{2}}_n(\cos(\theta)) = -\tfrac{n+d-1}{2}\sin(\theta)P^{\frac{d}{2}, \frac{d}{2}}_{n-1}(\cos(\theta)),
\end{equation}
from which, using Lemma \ref{Lemma: Generic Bounds for Jacobi Polynomails},   we find the following bounds   for $Y_n$ and $\partial_\theta Y_n$ 
\begin{align}\label{Eqution: Derivative Asymptotics}
|Y_n|\les  \frac{n^{\frac{d-1}2}}{ \la \theta n\ra^{\frac{d-1}2}},\,\,\,\,|\partial_\theta Y_n|\les \frac{n^{\frac{d+3}2}\theta }{ \la \theta n\ra^{\frac{d+1}2}}.
\end{align}
Applying integration by parts we see
\begin{align}
\int_{\frac1n}^{\frac\pi2}Y_{n_1}Y_{n_2}\omega(\theta;n,d)\,d\theta &= Y_{n_1}(\pi/2)Y_{n_2}(\pi/2)\int_{\frac1n}^{\frac\pi2}\omega(s; n, d)\,ds \label{Equation: Smoothing Lemma after IBP}\\
&\qquad-\int_{\frac1n}^{\frac\pi2}\partial_\theta\left(Y_{n_1}Y_{n_2}\right)\int_{\frac1n}^{\theta}\omega(s; n, d)\,ds\,d\theta.\nonumber
\end{align}

By \eqref{Equation: Smoothing w bound}, we bound this by
$$
\les \frac1{ n^{1-}}+ \frac{n_1^{\frac{d-1}2}n_2^{\frac{d-1}2}}{ n^{1-}}\int_{\frac1n}^{\frac\pi2}   \frac{  n_1 }{\la \theta n_1\ra^{\frac{d-1}2}\la \theta n_2\ra^{ \frac{d-1}2}} d\theta \les \frac{(n_1n_2)^\frac{d-1}{2}}{n^{1-}}.
$$
To obtain the last inequality, consider  the integrals on $(\frac1n,\frac1{n_1})$, $(\frac1{n_1},\frac1{n_2})$, $(\frac1{n_2},\frac\pi2)$ separately.
\end{proof}

Writing the nonlinearity, $ |u|^2u$,  on the Fourier side, $u=\sum_{n=0}^\infty \widehat{u}_n Y_n$, we have 
 \[ \mF(|u|^2u) (n)=      \sum_{n_1,n_2,n_3}\widehat{u}_{n_1}\overline{\widehat{u}_{n_2}}\widehat{u}_{n_3}\kappa(n_1,n_2,n_3,n)  ,
\]
where $\kappa$ is as in \eqref{kappa}.  
We split the resonant portion, $n_1=n$ or $n_3=n$, into
\begin{align*}
   2 \widehat{u}_n\sum_{n_2,n_3}\overline{\widehat{u}_{n_2}}\widehat{u}_{n_3}\kappa(n,n,n_2,n_3) -  \widehat{u}_n^2\sum_{n_2}\overline{\widehat{u}}_{n_2}\kappa(n,n,n_2,n),
\end{align*}
where the second term is not only, of course, easy to handle down to the local well-posedness level, but also presents another large frequency to aid in smoothing. The first, however, is of the form
\[
\widehat{u}_n\frac{2}{\pi\omega_d}\sum_{n_2,n_3}\overline{\widehat{u}_{n_2}}\widehat{u}_{n_3}\int_0^\pi Y_{n_2}(\theta)Y_{n_3}(\theta)\,d\theta + \mbox{smoother},
\]
by Lemma \ref{Lemma: Resonance Bound}.  
The first of these terms is analagous to the $\|u\|_{L^2_x}$ term that appears in the smoothing statement of \cite{ErTz2}, in that it is $\widehat{u}_n$ multiplied by a \textit{real} function of time and the solution.
Motivated by this, we define the change of variables given by
\begin{equation}\label{Definition: V variable}
    \widehat{u}_n = e^{\pm i \gamma(t; v)  }\widehat{v}_n, \,\,\,\,\text{ where} 
\end{equation}
\begin{equation}\label{Equation: Phase Rotation Removal}
\gamma(t; v) = \frac{2}{\pi\omega_d}\sum_{k,\ell}\overline{\widehat{v}_{k}}(t)\widehat{v}_{\ell}(t)\int_{0}^\pi Y_{k}(\theta)Y_{\ell}(\theta)\,d\theta,
\end{equation}
which acts as a phase rotation dependent only on time and the solution. Moreover, because of the conjugate in the definition we see that this is easily invertible.  
The resulting equation is then given by
\begin{equation}\label{phasedout}
    \begin{cases}
     i\partial_t v+\bigtriangleup_{\mathbb{S}^d}v \pm   v(|v|^2-\gamma(t;v)) = 0\\
     v(x,0) = u_0(x).
    \end{cases}
\end{equation}
Before we consider the Fourier transform of the new nonlinearity, $N(v):= \pm   v(|v|^2-\gamma(t;v))$,  we consider the sets given in Lemma~\ref{Lemma: Symbol decomposition} and define:
\begin{align*}
\Lambda_0(n) &= \{(n_1, n_2, n_3)\in \mathbb{N}_0^3\,: n_1 = n \mbox{ or }n_3 = n  \}\\
\Lambda_1(n) &= \{(n_1, n_2, n_3)\in \mathbb{N}_0^3\setminus\Lambda_0(n)\,: \la n_1\ra \la n_2\ra \la n_3\ra\gtrsim n^{3/2}  \}\\
\Lambda_2(n) &= \{(n_1, n_2, n_3)\in \mathbb{N}_0^3\setminus\left(\Lambda_0(n)\cup\Lambda_1(n)\right)\,: |H_n|\gtrsim \max(n_1,n_2,n_3) |n-\max (n_1,n_3)| \}.
\end{align*}
We note that these sets are disjoint for all $n\in\mathbb{N}_0$ and that these sets directly correspond to frequency configurations highlighted in Lemma \ref{Lemma: Symbol decomposition}. In particular,  the set $\Lambda_2(n)$ contains   the indices when the phase, $H_n$, is large.
In order to use the fact that we have large modulation, we need to apply differentiation by parts for the contribution of these terms.  
Before proceeding, we note that we've truncated the summation notation, opting to drop the $(n_1, n_2,n_3)\in\Lambda_i(n)$ in favor of simply stating $\Lambda_i(n)$, as there is no confusion.

The Fourier coefficients of the nonlinearity, $N(v):= \pm   v(|v|^2-\gamma(t;v))$, are given by 
\begin{multline*}  
\widehat{N(v)}(n)=  
 \pm 2\widehat{v}_n\sum_{n_2,n_3}\overline{\widehat{v}_{n_2}}\widehat{v}_{n_3}\left(\kappa(n,n,n_2,n_3) - \frac{1}{\pi\omega_d}\int_0^\pi Y_{n_2}(\theta)Y_{n_3}(\theta)\,d\theta\right) \\
 \mp \widehat{v}_n^2\sum_{n_2}\overline{\widehat{v}}_{n_2}\kappa(n,n,n_2,n)  
 \pm \sum_{\Lambda_1(n)}\widehat{v}_{n_1}\overline{\widehat{v}_{n_2}}\widehat{v}_{n_3}\kappa(n,n_1,n_2,n_3) 
 \pm \sum_{\Lambda_2(n)}\widehat{v}_{n_1}\overline{\widehat{v}_{n_2}}\widehat{v}_{n_3}\kappa(n,n_1,n_2,n_3)\\
 =: \widehat{N_{0,1}(v)}(n)+\widehat{N_{0,2}(v)}(n) +\widehat{N_1(v)}(n) +\widehat{N_2(v)}(n). 
\end{multline*}
By differentiation by parts, applied as in \cite[Proposition 6.1]{EGT} only to $N_2$, the solution of \eqref{phasedout} satisfies
 \be\label{preduhamelR}
i\partial_t \left(e^{-it \Delta } v - e^{-it \Delta }  B (v)  \right)= - e^{-it \Delta } \big(N_{0,1}(v) + N_{0,2}(v) + N_1(v)  +N_{2,1}(v)+N_{2,2}(v)  \big),
\ee
where
$$
\widehat{B (v)}(n)= \pm i \sum_{\Lambda_2(n)}\frac1{H_n} \widehat{v}_{n_1}\overline{\widehat{v}_{n_2}}\widehat{v}_{n_3}\kappa(n,n_1,n_2,n_3),  
$$
$$
\widehat{N_{2,1}(v)}(n)=\pm 2i \sum_{\Lambda_2(n)}\frac1{H_n} \widehat{w}_{n_1}\overline{\widehat{v}_{n_2}}\widehat{v}_{n_3}\kappa(n,n_1,n_2,n_3), 
$$
$$
\widehat{N_{2,2}(v)}(n)=\mp  i \sum_{\Lambda_2(n)}\frac1{H_n} \widehat{v}_{n_1}\overline{\widehat{w}_{n_2}}\widehat{v}_{n_3}\kappa(n,n_1,n_2,n_3).
$$
Here 
$$
w=ie^{it\Delta}[\partial_t(e^{-it\Delta}v)]=-  N(v) = \mp   v(|v|^2-\gamma(t;v)).
$$
A priori estimates for the terms $B(v)$, $N_1(v) $, $N_{2,1}(v)$, and $N_{2,2}(v) $ will be given in Lemma~\ref{Lemma: Main Smoothing Lemma} below. The term $N_{0,1}(v)$ will be estimated in Lemma~\ref{Lemma: Single Resonant Smoothing Lemma}, and the term $N_{0,2}(v)$ in Lemma~\ref{Lemma: Double Resonant Smoothing Lemma}.

The following proposition is a repeatedly used application of Young's inequality (or just Cauchy-Schwarz), which is best to state once and use by reference.
 \begin{proposition}\label{proposition: General Lemma for Gamma}
Let  $0\leq \delta\leq 1$. Then for any $\{a_n\}$, $\{b_m\}$ we have
for any $\eta> \tfrac{1-\delta}2$,\[
\Big|\sum_{n,m\geq 0}\frac{a_nb_m}{\langle n-m\rangle^\delta}\Big|\lesssim \|a_n \la n\ra^{\eta} \|_{\ell^2_n}  \|b_m \la m\ra^{\eta} \|_{\ell^2_m}.  
\]
\end{proposition} 

Before proceeding, we define the space within which the wellposedness arguments are done-- the standard Bourgain space adapted to $\mathbb{S}^d$:
\begin{align*}
\|u\|_{X^{s,b}} &:= \|\langle n\rangle^{s}\langle \tau+n(n+d-1)\rangle^b|\mathcal{F}_{x,t}u(n,\tau)|\|_{L^2_\tau\ell^2_n},\\
\|u\|_{X^{s,b}_T} &:= \inf_{w|_{[0,T]}=u|_{[0,T]}}\|w\|_{X^{s,b}}.
\end{align*}

An easy consequence of the definition of these spaces and Proposition \eqref{proposition: General Lemma for Gamma} is that the phase rotation \eqref{Definition: V variable} is well defined. That is, $\int_0^t\gamma(s,v)\,ds$ is finite for $v\in X^{\frac{d-2}{2}+,1/2+}$: 
\begin{proposition}\label{prop:gammafinite}
For $0<t<T$, we have 
    \begin{equation}\label{Equation: Proposition Gamma integral to bound}
        \big|\int_0^t \sum_{k,\ell}\overline{\widehat{v}_{k}}(s)\widehat{v}_{\ell}(s)\int_{0}^\pi Y_{k}(\theta)Y_{\ell}(\theta)\,d\theta\,ds\big| \les \|v\|_{X^{\frac{d-2}{2}+,1/2+}}^2,
    \end{equation}
where the implicit constant depends on $T$ only. 
\end{proposition}
\begin{proof}
We first note that by \eqref{Eqution: Derivative Asymptotics} we have $\int_{0}^\pi Y_{k}(\theta)Y_{\ell}(\theta)\,d\theta = O((k\ell)^{\frac{d-2}{2}+})$. Therefore the contribution of the terms $k=\ell$ is bounded by $T \|v\|_{C^0_{t\in[0,T]}H^{\frac{d-2}{2}+}_x}^2$, which suffices. For the terms $k\neq \ell$, by Plancherel, and noting that $|\widehat{\chi_{[0,t]}}(\tau)|\les \frac 1{\la \tau\ra}$ with an implicit constant depending on $T$ only, we have the bound
\begin{multline*}
\sum_{k\neq \ell}\int_{\R^2} \frac{|\mF_{x,t}v (\ell,\tau_1)| |\mF_{x,t}v(k,\tau)|(k\ell)^{\frac{d-2}{2}+}}{\la \tau-\tau_1\ra} d\tau_1d\tau \\
\les \|v\|_{X^{\frac{d-2}{2}+,1/2+}}^2 \Big[\sum_{k\neq \ell}\int_{\R^2}\frac{d\tau_1d\tau }{\la \tau-\tau_1\ra^2\la \tau- k(k+d-1)\ra^{1+} \la \tau_1-\ell(\ell+d-1)\ra^{1+}}\Big]^{1/2}\\
\les  \|v\|_{X^{\frac{d-2}{2}+,1/2+}}^2 \Big[\sum_{k\neq \ell} \frac{1 }{ \la   k(k+d-1) -\ell(\ell+d-1)\ra^{1+}}\Big]^{1/2}
 \les  \|v\|_{X^{\frac{d-2}{2}+,1/2+}}^2.
\end{multline*}
In the second inequality we used Cauchy-Schwarz in all variables and the definition of $X^{s,b}$ norm. 
\end{proof}

We also need another proposition which is a bilinear Strichartz estimate that follows from \eqref{Lemma: Improved Bilinear Strichartz} and the bilinear form of \eqref{Equation: Sphere Strichartz}, see \cite[Proposition 4.3]{BurqMultilinearStrichartz}.

\begin{proposition}\label{Proposition: General Bilinear Strichartz}
 For $N\geq M$ dyadic and all $\varepsilon > 0$, we have 
\[
\|P_N(\eta)P_M(\nu)\|_{L^2_{t\in[0,2\pi]} L^2_x} \lesssim M^{\frac{d-2}{2}+\varepsilon}\|P_N(\eta)\|_{X^{0, 1/2-}}\|P_M(\nu)\|_{X^{0, 1/2-}},
\]
and hence for any $ \alpha,\beta\geq 0$ satisfying $\alpha+\beta>\frac{d-2}2$, we have
\[\| \eta \nu \|_{L^2_{t\in[0,2\pi]} L^2_x}\lesssim \| \eta \|_{X^{\alpha, 1/2-}}\|  \nu \|_{X^{\beta, 1/2-}}\]
\end{proposition}

The rest of this section consists of  the required estimates for the terms appearing in \eqref{preduhamelR}.
\begin{lemma}\label{Lemma: Main Smoothing Lemma}
Let $d\geq 2$, $s > \frac{d-2}{2}$,   and $
0 \leq \varepsilon  < \frac{1}{2}\min\left(s-\tfrac{d-2}{2}, 2\right)$, then
\be \label{firstbound} 
    \left\|N_{2,1}(v) \right\|_{X^{s+\varepsilon,-1/2+}_T}+ \left\|N_{2,2}(v) \right\|_{X^{s+\varepsilon,-1/2+}_T} 
     \lesssim_\varepsilon   \|v\|_{X^{s,1/2+}_T}^{5},
\ee
\be \label{secondbound}
     \left\|N_{ 1}(v)\right\|_{X^{s+\varepsilon,-1/2+}_T} 
 \lesssim_\varepsilon \|v\|_{X^{s,1/2+}_T}^{3},
\ee 
\be \label{thirdbound}
    \left\|B(v) \right\|_{C^0_tH^{s+\varepsilon}_x}  \lesssim_\varepsilon \|v\|_{C^0_tH^s_x}^{3}.
\ee
\end{lemma}
\begin{proof}
We first handle the term $N_{2,1}(v)$, neglecting the term $N_{2,2}(v)$ as it is proved similarly. We then split into when we have $v|v|^2$ and $v\gamma(t;v)$, separately.  

When we have $v|v|^2$,  using the restriction imposed by $\Lambda_2(n)$, we have 
\[
|H_n|\gtrsim \max(n_1, n_2, n_3)|n-\max(n_1,n_3)|.
\]
Similarly, we find by the disjointness of $\Lambda_i(n)$ for $i\in \{1, 2\}$ and the fact that at least one $n_j\gtrsim n$ for $j\in\{1, 2, 3\}$ that $\la n_{i_1}\ra \la n_{i_2}\ra \ll n^{1/2}$ for some $\{i_1, i_2\}\subset\{1, 2, 3\}$ and $i_1\ne i_2$. We now ignore the presence of conjugates\footnote{This doesn't create any issues since in   Proposition~\ref{Proposition: General Bilinear Strichartz} we can replace $\eta$ and/or $\nu$ with their conjugates on the left hand side of the inequalities.  } and consider two cases:
\begin{align*}
\mbox{i) } \{n_2, n_3\} = \{i_1, i_2\}\qquad\mbox{ii) } \{n_1, n_3\} = \{i_1, i_2\}.
\end{align*}
In the first case we find that $n_1\sim n$, and hence $|H_n|\gg n_1\gtrsim \la n_2\ra\la n_3\ra$. Relabeling 
\begin{align*}
\mathcal{F}_{x,t}(\varphi)(n, \tau) = |\mathcal{F}_{x,t}(v)(n, \tau)|\qquad \mbox{ and }\qquad\mathcal{F}_{x,t}(\nu)(n,\tau)= \mathcal{F}_{x,t}\left(\varphi^3\right)(n,\tau)
\end{align*}
and noting that 
\[
\frac{\la n\ra^{\varepsilon}\la n_2\ra^{\frac{d}2+}\la n_3\ra^{\frac{d}2+}}{|H_n|\la n_2\ra^s\la n_3\ra^s}\lesssim \la n_2\ra^{\frac{d}2-s-1+\varepsilon+}\la n_3\ra^{\frac{d}2-s-1+\varepsilon+}\lesssim 1
\]
for $\varepsilon < \max\left(s-\tfrac{d-2}{2}, 1\right)$, we see that it is sufficient bound
\[
    \Big\|\mathcal{F}^{-1}_n\Big(\sum_{\Lambda_2(n)}\la n_1\ra^{s}\widehat{\nu}_{n_1}\la n_2\ra^{s-\frac{d}{2}+}\widehat{\varphi}_{n_2}\la n_3\ra^{s-\frac{d}{2}+}\widehat{\varphi}_{n_3}\kappa(n,n_1,n_2,n_3)\Big)\Big\|_{X^{0,-1/2+}_T}.
\]
By using the positivity of the space-time Fourier transforms, we may expand the summation from $\Lambda_2(n)$ to $\mathbb{N}_0^3$, use
$$
\la n_1\ra^{s}\widehat{\nu}_{n_1}\les \widehat{[\varphi^2 J_x^s\varphi]}_{n_1},
$$
 invoke duality for $w\in X^{0, 1/2-}$, and apply Parseval's to find
\begin{multline*}
    \bigg\|\mathcal{F}^{-1}_n\bigg(\sum_{(n_1,n_2,n_3)\in\mathbb{N}_0} \la n_1\ra^{s}\widehat{\nu}_{n_1}\la n_2\ra^{s-\frac{d}{2}+}\widehat{\varphi}_{n_2}\la n_3\ra^{s-\frac{d}{2}+}\widehat{\varphi}_{n_3}\kappa(n,n_1,n_2,n_3)\bigg)\bigg\|_{X^{0,-1/2+}_T}\\
    \lesssim \int_{\mathbb{R}}\int_{\mathbb{S}^d}\big|w \varphi^2 J^s_x\varphi \big(J_x^{s-\frac{d}{2}+}\varphi\big)^2\big|\,dtd\sigma(x)\lesssim   
    \|w\varphi^2J^s_x\varphi\|_{L^1_{x,t}}\|J^{s-\frac{d}{2}+}\varphi\|_{L^\infty_{t,x}}^2\\
    \lesssim \|w\varphi\|_{L^2_{x,t}}\|\varphi J^s_x\varphi\|_{L^2_{x,t}}\|\varphi\|_{X^{s, 1/2+}_T}\lesssim \|v\|_{X^{s,1/2+}_T}^5,
\end{multline*}
by the bilinear $L^2$ estimate (Proposition~\ref{Proposition: General Bilinear Strichartz}) and Sobolev embedding. 

In the second case,  $\{n_1, n_3\} = \{i_1, i_2\}$, we have $n_2\sim n$ and $n_2^{1/2}\gtrsim \la n_1\ra \la n_3\ra$. In particular, $|H_n|\gtrsim n_2$ (in fact we have $|H_n|\gtrsim n_2^2$, but we only use $n_2$ in order for the argument to handle the case that $n_3\sim n$). Therefore
\[
\frac{\langle n\rangle^{\varepsilon}}{|H_n|} \lesssim \frac{1}{n_2^{1-\varepsilon}}\lesssim \frac{1}{\la n_1\ra ^{2-2\varepsilon}\la n_3\ra^{2-2\varepsilon}},
\]
we expand the summation from $\Lambda_2(n)$ to $\mathbb{N}^3_0$, and find:
\begin{align}\label{nuhl}
    \lesssim\bigg\|\mathcal{F}^{-1}_n\bigg(\sum_{(n_1,n_2,n_3)\in\mathbb{N}_0}  \la n_1\ra^{2\varepsilon-2}\widehat{\nu}_{n_1}\la n_2\ra^s\widehat{\varphi}_{n_2}\la n_3\ra^{2\varepsilon-2}\widehat{\varphi}_{n_3}\kappa(n,n_1,n_2,n_3)\bigg)\bigg\|_{X^{0,-1/2+}_T} 
    \end{align}
 Note that
 \[
\widehat{\nu}_{n } =  \sum_{n_{1 }, n_{ 2}, n_{ 3} }\widehat{\varphi}_{n_1}\widehat{\varphi}_{n_2}\widehat{\varphi}_{n_3}\kappa(n, n_{1},n_{2}, n_{3}).
\]  
By the support condition of $\kappa$ and symmetry, it suffices to consider the cases $n\ll n_1\approx n_2\gtrsim n_3$ and $n\approx n_1 \gtrsim n_2,n_3$. Denoting the contributions of these terms as $\nu^h$, $\nu^\ell$, respectively, we see that
\be\label{nuh}
\mF(J_x^\alpha \nu^h)(n) \les \mF[(J_x^{\alpha/2} \varphi)^2 \varphi](n),\,\,\,\,\text{ for } \alpha\geq 0, \text{ and}
\ee
\be\label{nul}
\mF(J_x^\alpha \nu^\ell)(n) \les \mF[(J_x^{\alpha } \varphi)  \varphi^2](n),\,\,\,\,\text{ for all } \alpha\in \R.
\ee
Using these we estimate the contributions of $\nu^h$ and $\nu^\ell$ to \eqref{nuhl} as follows. For $\nu^h$, using \eqref{nuh} with $\alpha=d+2\epsilon-2+$, and duality, it suffices  to bound
\begin{multline*}
\int_{\mathbb{R}}\int_{\mathbb{S}^d} \big|w J_x^{-d-} [\varphi(J_x^{\varepsilon+\frac{d-2}2+}\varphi)^2] J_x^s\varphi J_x^{2\varepsilon-2}\varphi \big|\,dtd\sigma(x) \\
 \leq \|wJ_x^{2\varepsilon-2}\varphi\|_{L^2_{x,t}} \big\|J_x^{-d-} [\varphi(J_x^{\varepsilon+\frac{d-2}2+}\varphi)^2]\big\|_{L_t^2L_x^\infty} \|J_x^s\varphi \|_{L^\infty_tL^2_x} \\ \les \|w\|_{X^{0,1/2-}} \|\varphi\|_{X^{s,1/2+}}^2  \big\| \varphi(J_x^{\varepsilon+\frac{d-2}2}\varphi)^2 \big\|_{L_t^2L_x^1}.
\end{multline*}
In the last step we used the bilinear $L^2$  estimate and sobolev embedding. The following inequality  finishes the proof for the contribution of $\nu^h$:
$$
\big\| \varphi(J_x^{\varepsilon+\frac{d-2}2+}\varphi)^2 \big\|_{L_t^2L_x^1}\leq  \|\varphi  J^{\varepsilon+\frac{d-2}{2}+}_x\varphi \|_{L^2_{t,x}}\|J^{\varepsilon+\frac{d-2}{2}+}_x\varphi \|_{L^\infty_tL^2_x}\lesssim \|\varphi\|_{X^{s,1/2+}}^3.
$$ 
 Above we used the bilinear $L^2$  estimate  and that $\epsilon<s-\frac{d-2}2$.
 
For the contribution of $\nu^\ell$ to  \eqref{nuhl},    using \eqref{nuh} with $\alpha= 2\epsilon-2$,   duality, bilinear $L^2$  estimate, and Sobolev embedding, we have the required bound:   
\begin{multline*}
\int_{\mathbb{R}}\int_{\mathbb{S}^d} \big|w  \varphi^2 J_x^{2\varepsilon-2}\varphi   J_x^s\varphi J_x^{2\varepsilon-2}\varphi \big|\,dtd\sigma(x) \\
 \leq \|w \varphi\|_{L^2_{x,t}}  \| \varphi J_x^s\varphi \|_{L^2_{x,t}} \|J_x^{2\varepsilon-2}\varphi  \|_{L^\infty_{x,t}}^2  \les \|w\|_{X^{0,1/2-}} \|\varphi\|_{X^{s,1/2+}}^5.  
\end{multline*}

We now consider the contribution of $v\gamma(t;v)$ to $N_{2,1}(v)$.  Recall that
$$
\gamma(t; v) = \frac{2}{\pi\omega_d}\sum_{k,\ell}\overline{\widehat{v}_{k}}(t)\widehat{v}_{\ell}(t)\int_{0}^\pi Y_{k}(\theta)Y_{\ell}(\theta)\,d\theta.
$$
We first need to understand the Fourier transform of $\gamma$ in $t$. Defining $\varphi$ as above and noting that   \eqref{Eqution: Derivative Asymptotics} yields $\int_0^\pi Y_k Y_\ell\,d\theta = O((k\ell )^{\frac{d-2}{2}+})$, we have  (for each $n$)
\begin{multline*}
\big|\mathcal{F}_t(\gamma(t;v))(\tau)\big| \lesssim  \int \sum_{k \ne \ell < n}\mathcal{F}_{x,t}(\bar \varphi)(k, \tau_{ 1})\mathcal{F}_{x,t}(\varphi)(\ell, \tau-\tau_{1})(k\ell)^{\frac{d-2}{2}+}\,d\tau_1\\
+\int \sum_{\substack{\max(k, \ell)\geq n\\k\ne \ell}}\mathcal{F}_{x,t}(\bar \varphi)(k, \tau_{ 1})\mathcal{F}_{x,t}(\varphi)(\ell, \tau-\tau_{1})(k\ell)^{\frac{d-2}{2}+}\,d\tau_1 + \mathcal{F}_{t}(\|\varphi\|_{H^{\frac{d-2}{2}+}_x}^2)\\
=\mathcal{F}_t\Big( \sum_{k \ne \ell < n}\widehat{\bar \varphi}_k \widehat{  \varphi}_\ell (k\ell)^{\frac{d-2}{2}+} +
\sum_{\substack{\max(k, \ell)\geq n\\k\ne \ell}}\widehat{\bar \varphi}_k \widehat{  \varphi}_\ell (k\ell)^{\frac{d-2}{2}+} + \|\varphi\|_{H^{\frac{d-2}{2}+}_x}^2 \Big)(\tau) \\ 
=: \mathcal{F}_t\big(\gamma_{1,n} + \gamma_{2,n} +\|\varphi\|_{H^{\frac{d-2}{2}+}_x}^2\big)(\tau)=:\mathcal{F}_t (\Gamma_n)(\tau),
\end{multline*}
where $\mathcal{F}_{x,t}(\bar \varphi)(k, \tau_{ 1})=\mathcal{F}_{x,t}(  \varphi)(k, -\tau_{ 1}) \geq 0$. 
With this bound the contribution of $v\gamma(t;v)$ boils down to estimating
\begin{multline}\label{bigmess}
    \Big\|\mathcal{F}^{-1}_n\Big(\sum_{\Lambda_2(n)}\frac{ \la n\ra^{s+\varepsilon} \Gamma_n}{|H_n|} \widehat{\varphi}_{n_1} \widehat{\bar \varphi}_{n_2}\widehat{\varphi}_{n_3}\kappa(n,n_1,n_2,n_3)\Big)\Big\|_{X^{0,-1/2+} }=\\
\Big\|  \int\sum_{\Lambda_2(n)}\frac{\la n\ra^{ s+\varepsilon} \mathcal{F}_t(\Gamma_n)(\tau_0)  \mF_{x,t}\varphi(n_1,\tau_1)\mF_{x,t}\bar\varphi(n_2,\tau_2)\mF_{x,t}\varphi(n_3,\tau_3) }{ \la \tau+n(n+d-1)\ra^{\frac12-} |H_n|}
 \kappa(n,n_1,n_2,n_3) d\tau_1d\tau_2 d\tau_3 \Big\|_{L^2_\tau\ell^2_n},
\end{multline}
where $\tau_0=\tau -\tau_1-\tau_2-\tau_3$.

For the contribution of $ \|\varphi\|_{H^{\frac{d-2}{2}+}_x}^2$ to \eqref{bigmess}, noting that $n\les \max(n_1,n_2,n_3)$, $\frac{  n^\varepsilon}{|H_n|}\les 1$, and using duality,
it suffices to observe that
$$
\big\| w \varphi^2 J^s\varphi \|\varphi\|_{H^{\frac{d-2}{2}+}_x}^2 \big\|_{L^1_{x,t}}\les \|w\varphi\|_{L^2_{x,t}}   \|\varphi  J^s\varphi \|_{L^2_{x,t}} \|\varphi\|_{C^0_tH^{\frac{d-2}{2}+}_x}^2 \les \|w\|_{X^{0,1/2-}}\|\varphi\|_{X^{s,\frac12+}}^5. 
$$
For the contribution of $\gamma_{1,n}$, first note that (since $k,\ell\leq n$ and $|H_n|\gtrsim n$)
\[
\frac{n^\varepsilon(k\ell)^{\frac{d-2}{2}+}}{|H_n|}  \les  (k\ell)^{s-\frac{1}{2}-},
\]
for $\varepsilon <  \frac{1}{2}\min\left(s-\tfrac{d-2}{2}, 2\right)$. Therefore, using Proposition~\ref{proposition: General Lemma for Gamma} in the $k,\ell$ sums with $\delta=0$, we see that 
\[\frac{ \la n\ra^{ \varepsilon} \Gamma_n}{|H_n|} \les \|\varphi\|_{H^s_x}^2.\]
The bound then is identical to the contribution of $ \|\varphi\|_{H^{\frac{d-2}{2}+}_x}^2$  above.   

To handle the contribution of the $\gamma_{2,n}$ term, we consider two subcases:
\[
|H_n|\gtrsim (k+\ell+d-1)|k-\ell|\quad \mbox{and}\quad|H_n|\ll (k+\ell+d-1)|k-\ell|.
\]
In the first case, we recall that $\max(k,\ell)\gtrsim n$, $k\neq \ell$,  and hence
\begin{equation}\label{Equation: gamma2n first subcase}
\frac{\langle n\rangle^\varepsilon(k\ell)^{\frac{d-2}{2}}}{|H_n|}\lesssim \frac{(k\ell)^\frac{d-2}{2}}{|k-\ell|}
\end{equation}
for $\varepsilon < 1$, so that by Proposition \ref{proposition: General Lemma for Gamma} with $\delta = 1$ we find
\begin{equation}
    \frac{\langle n\rangle^\varepsilon\Gamma_n}{|H_n|}\lesssim \|\varphi\|_{H^{s}_x}^2,
\end{equation}
for $s> \frac{d-2}{2}$. We see that the bound is again identical to the prior contributions.

When $|H_n|\ll (k+\ell+d-1)|k-\ell|$, we note that 
\[
|H_n+k(k+d-1)-\ell(\ell+d-1)|\gtrsim (k+\ell)|k-\ell|, 
\]
and hence by considering the weights in the $X^{s,b}$ norms %$b$ variable 
in \eqref{bigmess} we see
\begin{multline*}
    \sigma := \la \tau + n(n+d-1)\ra +\sum_{i=1}^3\la \tau_i+n_i(n_i+d-1)\ra + \la \lambda_{1}+k(k+d-1)\ra\\
+\la \lambda_{2}+\ell(\ell+d-1)\ra\gtrsim (k+\ell)|k-\ell|,      
\end{multline*}
so that one of the summands on the left hand side of the above must be $\gtrsim (k+\ell)|k-\ell|$. From this, we split into four subcases:
\begin{itemize}
\item[i)] $\la \tau + n(n+d-1)\ra\gtrsim (k+\ell)|k-\ell|$,
\item[ii)] $\la \tau_i+(-1)^in_i(n_i+d-1)\ra \gtrsim (k+\ell)|k-\ell|$ for $1\leq i\leq 3$,
\item[iv)] $\la \lambda_1-k(k+d-1)\ra \gtrsim (k+\ell)|k-\ell|$,
\end{itemize}
and let $\sigma_0 = \la \tau + n(n+d-1)\ra$ and $\sigma_i = \la \tau_i+(-1)^in_1(n_1+d-1)\ra$ for $1\leq i\leq 3$. We also note that for $\varepsilon < \frac{1}{2}(s-\frac{d-2}{2})$ we have
\begin{equation}\label{Equation: Generic Modulation weight Bound}
\la n\ra^\varepsilon(k\ell)^{\frac{d-2}{2}+}\lesssim \sigma^{1/2-}\frac{(k\ell)^{s-\frac{1}{4}-}}{|k-\ell|^{1/2-}}.
\end{equation}
To handle the first two cases we then have by Proposition~\ref{proposition: General Lemma for Gamma} with $\delta = 1/2-$ that
\[
\la n\ra ^\varepsilon \Gamma_n \lesssim \big(\sum_{i=0}^3\sigma_i\big)^{1/2-}\|\mF_t{\varphi}\|_{H^{s}_x}^2,
\]
and hence by Young's and the $\Lambda_2(n)$ restriction it is sufficient to bound
\begin{align}\label{Equation: bigmess reworked}
    \Big\|  \int\sum_{\Lambda_2(n)}\frac{\la n_1\ra^{s} \sigma_i^{1/2-} \mF_{x,t}\varphi(n_1,\tau_1)\mF_{x,t}\bar\varphi(n_2,\tau_2)\mF_{x,t}\varphi(n_3,\tau_3)}{ \la \tau+n(n+d-1)\ra^{\frac12-} \la n_2\ra \la n_3\ra}
    \kappa(n,n_1,n_2,n_3) d\tau_1d\tau_2 d\tau_3 \Big\|_{L^2_\tau\ell^2_n},
\end{align}
for $0\leq i\leq 3$. Note that we have harmlessly assumed that $n_1\gtrsim n_2,n_3$ in the above display. 

We now use Plancherel, duality with $w\in X^{0,1/2-}$, and H\"olders to find
\begin{multline*}
    \eqref{Equation: bigmess reworked}\lesssim \|\sigma_0 w\|_{L^2_{x,t}}\|J^s_x\varphi\|_{L^2_{x,t}}\|J^{-1}_x\varphi\|_{L^\infty_{x,t}} +\|w\|_{L^2_{x,t}}\|J^s_x\sigma_1\varphi\|_{L^2_{x,t}}\|J^{-1}_x\varphi\|_{L^\infty_{x,t}}\\ + \|w\|_{L^2_{x,t}}\|J^s_x\varphi\|_{L^2_{x,t}}\|J^{-1}_x\sigma_2\varphi\|_{L^2_tL^\infty_x}\|J^{-1}_x\varphi\|_{L^\infty_{x,t}}\lesssim \|\varphi\|_{X^{s,1/2+}}^3,
\end{multline*}
for $s > \frac{d-2}{2}$.

To handle the last case, we note again that by Youngs, \eqref{Equation: Generic Modulation weight Bound}, and proposition \ref{proposition: General Lemma for Gamma} with $\delta = 1/2-$, that we have
\begin{align}\label{Equation: Gamma term final modulation}
\bigg\|\la n\ra ^\varepsilon &\sum_{\substack{\max(k,\ell)\geq n\\k\ne\ell}}\int\mF_{x,t}\bar\varphi(k,\lambda_1)\mF_{x,t}\varphi(\ell,\tau-\lambda_1)\,d\lambda_1\big|\int_0^\pi Y_k(\theta)Y_\ell(\theta)\,d\theta\big|\bigg\|_{L^{2-}_\tau}\\
&\lesssim \sum_{\substack{\max(k,\ell)\geq n\\k\ne\ell}}\|\la \tau - k(k+d-1)\ra^{1/2-}\mF_{x,t}\bar\varphi(k,\tau)\|_{L^{2-}_{\tau}}\|\mF_{x,t}\varphi(\ell,\tau)\|_{L^1_\tau}\frac{(k\ell)^{s-1/4-}}{|k-\ell|^{1/2-}}\nonumber\\
 &\lesssim \|\la k\ra^{s}\la \tau - k(k+d-1)\ra^{1/2+} \mF_{x,t}{\bar\varphi}(k,\tau)\|_{L^{2}_\tau\ell^2_k}\|\la \ell\ra^{s}\mF_{x,t}{\varphi}(\ell,\lambda_2)\|_{\ell^2_\ell L^1_\tau},\nonumber
\end{align}
for $s > \frac{d-2}{2}$ and $\varepsilon < 1/2(s-\frac{d-2}{2})$. In other words, the contribution to $\Gamma_n$ can be estimated by
\[
\|\la n\ra^\varepsilon \Gamma_n(\tau)\|_{\ell^\infty_nL^{2-}_\tau} \lesssim \|\varphi\|_{X^{s,1/2+}}^2.
\]
We now estimate \eqref{bigmess} by assuming that $n_1\sim n$ and using $H_n\gtrsim\la n_2\ra \la n_3\ra$ to write
\begin{multline*}
    \eqref{bigmess}\lesssim \|\frac{1}{\la \tau+n(n+d-1)\ra^{1/2-}}(f*_\tau\mF_t\Gamma_n)\|_{L^2_\tau\ell^2_{n}}\lesssim \|f*_\tau\mF_t\Gamma_n\|_{\ell^2_{n}L^{\infty-}_\tau}\\
    \lesssim \big\|\|f\|_{L^{2}_\tau}\|\Gamma_n\|_{L^{2-}_\tau}\big\|_{\ell^2_n}\lesssim\|f\|_{L^2_nL^{2}_\tau}\|\Gamma_{n}\|_{\ell^\infty_nL^{2-}_\tau}\lesssim \|f\|_{L^2_nL^{2}_\tau}\|\varphi\|_{X^{s,1/2+}}^2,
\end{multline*}
where
\[
f =  \int\sum_{\Lambda_2(n)}\la n_1\ra^s\mF_{x,t}\varphi(n_1,\tau_1)\frac{\mF_{x,t}\bar\varphi(n_2,\tau_2)}{\la n_2\ra}\frac{\mF_{x,t}\varphi(n_3,\tau_3)}{\la n_3\ra}\kappa(n,n_1,n_2,n_3) d\tau_1d\tau_2 d\tau_3.
\]
Owing to Plancherel and Sobolev embedding we find that 
\[
\|f\|_{L^2_{\tau}\ell^2_n}\lesssim \|\varphi\|_{X^{s,1/2+}}\|J_x^{-1}\varphi\|_{X^{\frac{d}{2}+,1/2+}}^2 = \|\varphi\|_{X^{s,1/2+}}\|\varphi\|_{X^{\frac{d-1}{2}+,1/2+}},
\]
and the full bound follows.

The bound \eqref{secondbound} follows from Lemma \ref{Lemma: Symbol decomposition} and the summation restriction on $\Lambda_1(n)$, as we'll have $\langle n_1\rangle \langle n_2\rangle \langle n_3\rangle \gtrsim n^{3/2}$. Specifically, we ignore the presence of conjugates because we will apply the bilinear $L^2$ estimate and and assume that either $n_1\sim n$ or $n_1\sim n_2\gg n$. In either case we must have that $\la n_2\ra \la n_3\ra \gtrsim n^{1/2}$. 

It then follows by \eqref{nul}, duality with $w\in X^{0,1/2-}$, and the bilinear $L^2$ estimate that
\[
\|J^s\varphi (J^{2\varepsilon}\varphi)\|_{X^{0,-1/2+}}\lesssim \|wJ^{2\varepsilon}\varphi\|_{L^2_{x,t}}\|J^s\varphi J^{2\varepsilon}\varphi\|_{L^2_{x,t}}\lesssim \|\varphi\|_{X^{s,1/2+}}\|\varphi\|_{X^{\frac{d-2}{2}+2\varepsilon+, 1/2+}}^2\lesssim \|\varphi\|_{X^{s,1/2+}}^3,
\]
 for $s > \frac{d-2}{2}$ and $\varepsilon < \frac{1}{2}(s-\frac{d-2}{2})$.

The bound \eqref{thirdbound} will follow from several applications of Cauchy-Schwarz. We assume $n_1\geq n_2, n_2$, so that the summation restriction of $\Lambda_2(n)$ implies that $\la n_2\ra \la n_3\ra\ll n^{1/2}$, hence 
\[
\frac{1}{(\la n_2\ra \la n_3\ra)^{s-\frac{d-2}{2}-}n^{1-\varepsilon}}\lesssim \frac{n^{0-}}{(\la n_2\ra \la n_3\ra)^{1/2+}}    
\]
for $\varepsilon < \frac{1}{2}\min(s-\frac{d-5}{2}, 2).$ Now, given $s > \frac{d-2}{2}$ and $\varepsilon < \frac{1}{2}\min(s-\frac{d-5}{2}, 2)$ we use Lemma~\ref{Lemma: Resonance Bound} to bound $\kappa$ by $O((n_2n_3)^{\frac{d-2}{2}+})$ and invoke the above display to find
\begin{multline*}
\|B(v)\|_{C^0_{t}H^{s+\varepsilon}_x}
\lesssim \big\|\sum_{\substack{n_1,n_2,n_3\\n\ne n_1,\,\Lambda_2(n)}}|\widehat{v}_{n_1}\widehat{\bar v}_{n_2}\widehat{v}_{n_3}|\frac{\la n\ra ^{s+\varepsilon}(n_2n_3)^{\frac{d-2}{2}+}}{n|n-n_1|}\big\|_{C^0_t\ell^2_n}
\\\lesssim\big\|\sum_{\substack{n_1,n_2,n_3\\n\ne n_1,\,\Lambda_2(n)}}|\widehat{v}_{n_1}\widehat{\bar v}_{n_2}\widehat{v}_{n_3}|\frac{\la n_1\ra^{s-} (\la n_2\ra \la n_3\ra)^{s-\frac{1}{2}-}}{|n-n_1|}\big\|_{C^0_t\ell^2_n} 
\\\lesssim \|u\|_{C^0_tH^{s}_x}^2\big\|\sum_{n\ne n_1}\frac{\la n_1\ra ^{0-}|\widehat{u}_{n_1}|}{|n-n_1|}\big\|_{C^0_t\ell^2_n}
\lesssim \|u\|_{C^0_tH^{s}_x}^3,
\end{multline*}
by Young's and Cauchy-Schwarz.
\end{proof} 
\begin{lemma}\label{Lemma: Single Resonant Smoothing Lemma}
Let $d\geq 2$, $s> \frac{d-2}{2}$,  and $0 \leq \varepsilon  < \min(s-\frac{d-2}{2}, 1)$,   then
\begin{align*}
    \left\|N_{0,1}(v)\right\|_{X^{s+\varepsilon,-1/2+}_T} 
    \lesssim_\varepsilon \|v\|_{X^{s,1/2+}_T}^3.
\end{align*}
\end{lemma}
\begin{proof}
We define
\begin{align*}
    \mathcal{F}_{x,t}(\varphi)(n, \tau) &:= |\mathcal{F}_{x,t}(v)(n, \tau)|,\\
    \tilde{\kappa}(n,n,n_2,n_3) &:=\kappa(n,n,n_2,n_3) - \frac{1}{\pi\omega_d}\int_0^\pi Y_{n_2}(\theta)Y_{n_3}(\theta)\,d\theta,
\end{align*}
so that we are reduced to bounding
\begin{multline*}
|\mF_{x,t}(N_{0,1}(v))(n,\tau)| 
\\\lesssim \int\mF_{x,t}(\varphi)(n,\tau_1)\sum_{n_2,n_3\leq n}\mF_{x,t}(\bar\varphi)(n_2,\tau_2)\mF_{x,t}(\varphi)(n_3,\tau-\tau_1-\tau_2)|\tilde{\kappa}(n,n,n_2,n_3)|\,d\tau_1d\tau_2\\
 + \int\mF_{x,t}(\varphi)(n,\tau_1)\sum_{\max(n_2,n_3)>n}\mF_{x,t}(\bar\varphi)(n_2,\tau_2)\mF_{x,t}(\varphi)(n_3,\tau-\tau_1-\tau_2)|\tilde{\kappa}(n,n,n_2,n_3)|\,d\tau_1d\tau_2\\
 :=\mF_{x,t}(N_{0,1}^\ell + N_{0,1}^h).
\end{multline*}
In order to handle $N^\ell_{0,1}$ we see that by Lemma \ref{Lemma: Resonance Bound} that we have
\begin{align*}
    \kappa(n,n,n_2,n_3) - \frac{1}{\pi}\int_0^\pi Y_{n_2}(\theta)Y_{n_3}(\theta)\,d\theta = O\left( \tfrac{(n_2n_3)^{\frac{d-1}{2}+}}{n}\right),
\end{align*}
and hence it suffices to bound 
\begin{align}\label{Equation: Single Resonance to Bound}
   \int\mF_{x,t}(\varphi)(n,\tau_1)\sum_{n_2,n_3\leq n}\mF_{x,t}(\bar\varphi)(n_2,\tau_2)\mF_{x,t}(\varphi)(n_3,\tau-\tau_1-\tau_3)\tfrac{(n_2n_3)^{\frac{d-1}{2}+}}{n}\,d\tau_1d\tau_2,
\end{align}
in $X^{s,1/2+}.$

Noticing that 
\[
|H_n(n, n_2, n_3, n)| = |(n_2+n_3+d-1)(n_2-n_3)| \gtrsim \max(n_2,n_3)|n_2-n_3|,
\]
we separate out two cases:
\begin{itemize}
    \item[I)] $n_2=n_3$
    \item[II)] $|H_n|\gtrsim \max(n_2, n_3)|n_2-n_3|.$
\end{itemize}

In the first case, we see that the $X^{s,1/2+}$ norm of \eqref{Equation: Single Resonance to Bound} reduces to bounding
\begin{multline*}
    \|\eqref{Equation: Single Resonance to Bound}\|_{X^{s,1/2+}}\lesssim   \bigg\|\la n\ra^{s+\varepsilon}\mF_x(\varphi)(n,t)\sum_{n_2\leq n}|\mF_x(\bar\varphi)(n_2,t)|^2\tfrac{n_2^{d-1+}}{n}\bigg\|_{L^2_t\ell^2_n}\\
    \lesssim \|\varphi\|_{L^\infty_tH^{s}_x}\|\varphi\|_{L^\infty_tH^{\frac{d-2+\varepsilon}{2}}_x}\|\varphi\|_{L^2_tH^{\frac{d-2+\varepsilon}{2}}_x}
    %\\&
    \lesssim \|\varphi\|_{X^{s,1/2+}_T}^3,
\end{multline*}
for $\varepsilon < \min\left(s-\frac{d-2}{2}, 1\right)$.

We now assume that we have modulation considerations at play. Specifically, if we are in case $II$ then we again find that 
\[
\la \tau_2+n_2(n_2+d-1)\ra +\la \tau_3 -n_3(n_3+d-1)\ra\gtrsim |H_n|\gtrsim \max(n_2,n_3)|n_2-n_3|,   
\]
and hence we may assume that $\langle \tau_2+n_2^2\rangle\gtrsim\max(n_2, n_3) |n_2-n_3|$. It follows that 
\[
\frac{\la n\ra^{\varepsilon}(n_2n_3)^{\frac{d-1}{2}+}}{n}\lesssim \frac{\la \tau_2+n_2(n_2+d-1)\ra^{1/2+}(n_2n_3)^{s-1/4-}}{|n_2-n_3|^{1/2+}},
\]
for $\varepsilon < \min(s-\frac{d-2}{2}, 1)$ and $s > \frac{d-2}{2}$, so that we find the contribution of the above to \eqref{Equation: Single Resonance to Bound} satisfies
\begin{multline*}
    \|\eqref{Equation: Single Resonance to Bound}\|_{X^{s,1/2+}}
     \lesssim \|\varphi\|_{L^\infty_tH^{s}_x}\\\times\big\|\int\sum_{n_2,n_3\leq n}\tfrac{\la \tau_2+n_2(n_2+d-1)\ra^{1/2+}(n_2n_3)^{s-1/4+}}{|n_2-n_3|^{1/2+}}\mF_{x,t}(\bar\varphi)(n_2,\tau_2)\mF_{x,t}(\varphi)(n_3,\tau-\tau_2)\,d\tau_2\big\|_{L^2_\tau \ell^\infty_n}\\
    \lesssim \|\varphi\|_{X^{s,1/2+}}^3,
\end{multline*}
by Proposition \ref{proposition: General Lemma for Gamma} with $\delta = 1/2+$. We note that the proof in the case that $\langle \tau_2+n_2^2\rangle\ll \max(n_2, n_3) |n_2-n_3|$ follows as above, with the only difference being which term is placed in $L^2_t$. This provides smoothing of order 
\begin{equation*}
0\leq \varepsilon < 2\min(s-\tfrac{d-2}{2},1/2),\mbox{ for }s > \tfrac{d-2}{2}.
\end{equation*}
We now assume that $\max(n_2, n_3)\gtrsim n$, so as to handle the contribution of $N_{0,1}^h$. By positivity and \eqref{Eqution: Derivative Asymptotics} we observe
\begin{equation}\label{Equation: Resonance simplification bound}
\tilde{\kappa}(n,n,n_2,n_3)\lesssim \kappa(n,n,n_2,n_3) + O((n_2n_3)^{\frac{d-2}{2}+})
\end{equation}
and 
\begin{multline*}
\la n\ra^\varepsilon \sum_{\max(n_2,n_3)\geq n}\mF_{x,t}(\bar\varphi)(n_2,\tau_2)\mF_{x,t}(\varphi)(n_3,\tau_3)\\
\lesssim  \sum_{\max(n_2,n_3)\geq n}\max(n_2,n_3)^\varepsilon\mF_{x,t}(\bar\varphi)(n_2,\tau_2)\mF_{x,t}(\varphi)(n_3,\tau_3),
\end{multline*}
so that the contribution to $N^h_{0,1}$ corresponding to the first term of \eqref{Equation: Resonance simplification bound} may be handled using the bilinear $L^2$ Strichartz estimate \ref{Proposition: General Bilinear Strichartz}. That is, the contribution of the above satisfies (by duality and the bilinear $L^2$ estimate)
\[
\|N^{h}_{0,1}\|_{X^{s+\varepsilon,1/2+}}\lesssim \|\varphi J^s\varphi J^\varepsilon \varphi\|_{X^{0,1/2+}}\lesssim \|\varphi\|_{X^{s,1/2+}}\|\varphi\|_{X^{\frac{d-2}{2}+\varepsilon, 1/2+}}^2\lesssim \|\varphi\|_{X^{s,1/2+}}^3,
\]
given $0 \leq \varepsilon < s-\frac{d-2}{2}$ and $s>\frac{d-2}{2}$. 

As for the contribution of the second term of \eqref{Equation: Resonance simplification bound}, we observe that 
\begin{multline*}
\la n\ra^\varepsilon \sum_{\max(n_2,n_3)\geq n}\mF_{x,t}(\bar\varphi)(n_2,\tau_2)\mF_{x,t}(\varphi)(n_3,\tau_3)(n_2n_3)^{\frac{d-2}{2}+}\\
\lesssim \sum_{\max(n_2,n_3)\geq n}\mF_{x,t}(\bar\varphi)(n_2,\tau_2)\mF_{x,t}(\varphi)(n_3,\tau_3)(n_2n_3)^{\frac{d-2}{2}+\varepsilon+}.
\end{multline*}
We may then bound the contribution to $N^h_{0,1}$ by using the exact same case work as that done to bound $N^{\ell}_{0,1}$. This yields smoothing of $0\leq \varepsilon < s-\frac{d-2}{2}$ for $s > \tfrac{d-2}{2}$.
\end{proof}

\begin{lemma}\label{Lemma: Double Resonant Smoothing Lemma}
Let $d\geq 2$, $s> \frac{d-2}{2}$, and  $0 \leq \varepsilon  \leq s-\frac{d-2}{2}$,  then
\begin{align*}
    \left\|N_{0,2}(v)\right\|_{X^{s+\varepsilon,-1/2+}_T}\lesssim_\varepsilon \|v\|_{X^{s,1/2+}_T}^3.
\end{align*}
\end{lemma}
\begin{proof}
This proof follows in exactly the same way as the proof of the prior lemma in the situation $\max(n_2, n_3)\gtrsim n$. In particular, if $\mathcal{F}_{x,t}(\varphi) = |\mathcal{F}_{x,t}(v)|$ then we find
\begin{multline}\label{Equation: Double resonance display}
    \la n\ra^{s+\varepsilon}\mF(\varphi)(n)^2\sum_{n_2}\mF(\bar\varphi)(n_2)\kappa(n,n,n_2,n)\\ = \la n\ra^{s}\mF(\varphi)(n)\la n\ra^{\varepsilon}\mF(\varphi)(n)\sum_{n_2}\mF(\bar\varphi)(n_2)\kappa(n,n,n_2,n),
\end{multline}
and hence by the bilinear $L^2$ estimate and duality with $w\in X^{0,1/2-}$:
\[
\|\eqref{Equation: Double resonance display}\|_{X^{0,1/2+}}\lesssim \|wJ_x^\varepsilon \varphi \|_{L^2_{x,t}}\|\varphi J^s\varphi\|_{L^2_{x,t}}\lesssim \|\varphi\|_{X^{s,1/2+}}^3,
\]
for $\varepsilon < s-\frac{d-2}{2}$ and $s > \frac{d-2}{2}$.
\end{proof}

\section{Appendix: Strichartz Estimates and the Cubic NLS}\label{Appendix: Strichartz}
In this appendix we establish a slightly strengthened bilinear Strichartz estimate for a class of functions in $\mathbb{S}^2$, which is useful for lower bounds on the fractal dimension of the graph of the free solution. As a corollary we establish an improved well-posedness statement for functions that are supported on the zonal harmonics that matches the statement on $\mathbb{T}^2$.

We recall the space $\mathcal{Z}^s$ and define $\mathcal{B}^s$ as
\begin{align*}
    \mathcal{Z}^s(\mathbb{S}^2) &:= \left\{f\in H^s(\mathbb{S}^2):\, f(\theta,\phi) = \sum_n a_n Y_n(\theta,\phi)\right\},\\
    \mathcal{B}^s(\mathbb{S}^2) &:= \left\{f\in H^s(\mathbb{S}^2):\, f(\theta,\phi) = \sum_{n}\sum_{j\in\{\pm 1\}} a_{nj}Y^{jn}_{n}(\theta,\phi)\right\}.
\end{align*}
The first of these spaces coincides with the space of functions in $H^s$ that are supported only on the zonal harmonics, whereas the second corresponds to the space of functions supported only on the gaussian beams $Y^{\pm n}_n$.

\begin{lemma}\label{Lemma: Improved Bilinear Strichartz}
Suppose that $f, g\in\mathcal{Z}^s$. Then for $N\geq M$ dyadic and every $\varepsilon > 0$ we have
\begin{equation}\label{Equation: Appendix Bilinear Strichartz}
\|P_N(e^{it\bigtriangleup_{\mathbb{S}^2}}f)P_M(e^{it\bigtriangleup_{\mathbb{S}^2}}g)\|_{L^2_{x,t\in[0,2\pi]}}\lesssim M^\varepsilon \|P_N(f)\|_{L^2_x}\|P_M(g)\|_{L^2_x}.
\end{equation}
\end{lemma}
\begin{proof}
We first assume that $f,g\in\mathcal{Z}^s$. Represent
\[
f = \sum_{n}f_nY_n,
\]
and similarly for $g$. Following \cite{BurqStrichartz}, we find by Parseval's (for $t\in[0,2\pi]$)
\begin{align}
\|P_N(e^{it\bigtriangleup_{\mathbb{S}^2}}f)P_M(e^{it\bigtriangleup_{\mathbb{S}^2}}g)\|_{L^2_{x,t}}^2 &= \sum_{\tau=0}^\infty\left\|\sum_{\substack{\tau = n(n+1)+m(m+1)\\ n\sim N,\,m\sim M}}f_ng_mY_nY_m\right\|_{L^2_x}^2\nonumber\\
&\leq \sum_{\tau = 0}^\infty \alpha_{NM}(\tau)\sum_{\substack{\tau = n(n+1)+m(m+1)\\ n\sim N,\,m\sim M}}|f_ng_m|^2\|Y_nY_m \|_{L^2_x}^2,\label{Equation: Strichartz Half way}
\end{align}
where 
\begin{equation}
    \alpha_{NM}(\tau) = \#\left\{\substack{(n,m)\in\mathbb{N}^2,\,n\sim N,\,m\sim M\\\tau = n(n+1)+m(m+1)}\right\}.
\end{equation}
By the divisor bound, we have that $\sup_\tau\alpha_{NM}(\tau)\lesssim M^\varepsilon$ for any $\varepsilon > 0$.

It then suffices to estimate the quantity $Y_nY_m$ in $L^2$. By the first bound in \eqref{Eqution: Derivative Asymptotics} we have
$$\|Y_nY_m\|_{L^2(\mathbb S^2)}^2\les \int_{0}^{\pi/2} \frac{\theta n m}{\la n\theta \ra \la m\theta\ra} d\theta \leq \int_{0}^{\pi/2} \frac{  m}{  \la m\theta\ra} d\theta \les M^\epsilon.
$$
Combining this with the bound for $\alpha_{NM}(\tau)$ and summing in $\tau$ we find
\[
\eqref{Equation: Strichartz Half way}\lesssim M^\varepsilon \|P_N(f)\|_{L^2_x}^2\|P_M(g)\|_{L^2_x}^2.
\]
For more details, see \cite{BurqStrichartz}.
\end{proof}

\begin{remark}\label{Remark: Appendix Saturation}
The above calculation isn't generic. That is, $f\in\mathcal{B}^s$ of the form $f = Y^{n}_n$ saturates the $L^4$ inequality on $\mathbb{S}^2$. Indeed, 
\[
Y^{n}_n = \frac{(-1)^n}{2^n n!}\sqrt{\frac{(2n+1)!}{4\pi}}\sin^n\theta e^{in\phi},
\]
so
\begin{align*}
    \|f\|_{L^4_{x,t}}^4\sim n\int_0^\pi\sin^{4n+1}\theta\,d\theta\sim n\frac{2^n(n!)^2}{(2n+1)!} \sim \sqrt{n},
\end{align*}
by Stirling's approximation. 

A similar calculation can be done to show that the Zonal harmonics essentially saturate the $L^4$ inequality for $d\geq 3$, \cite{BurqMultilinearStrichartz}.
\end{remark}

As a corollary of the above lemma, we find local well-posedness for $s > 0$ for functions on $\mathbb{S}^2$ that are independent of $\phi$. 
\begin{corollary}\label{Corollary: 2d zonal wellposed}
Let $s > 0$ and consider the space $\mathcal{Z}^s(\mathbb{S}^2)\subset H^s(\mathbb{S}^2)$. Then the equation
\begin{equation}\label{Equation: 3nls on sphere}
    \begin{cases}
    i\partial_t u + \bigtriangleup u \pm |u|^2u = 0\\
    u(x,0) = f(x)\in\mathcal{Z}^s(\mathbb{S}^2) 
    \end{cases}
\end{equation}
is locally well-posed for any $s > 0$.
\end{corollary}
\begin{remark}
In light of Remark \ref{Remark: Appendix Saturation} we see that the generic statement is $s > 1/4$ and cannot, in general, be improved.
\end{remark}

\section{Appendix: The Torus Case}\label{Section: Toral Bounds}\label{Section: Torus}
Before proceeding with theorem statements, we comment that Theorem~\ref{upperS} and Theorem~\ref{Spherical Deliu Jawerth} hold as stated for $\mathbb{T}^d$, and their proofs are standard. With preliminaries out of the way, we can prove a theorem analogous to Theorem \ref{Theorem: General Sphere Theorem} relatively easily for the Torus. In particular,
we obtain the following natural generalization of the one dimensional statement.

\begin{theorem}\label{Theorem: General TOrus}
Let $f: \mathbb{T}^{d} \to \mathbb{R}$ and define $f_N(x) := \sum_{N <\max_i\{|m_i|\} \leq 2N} \widehat{f}(m)e^{im\cdot x}$. Assume $f$ satisfies $\|f_N\|_{L^1_x} \lesssim N^{-(\frac{d}{2}+s)}$ for some $s \in (0,1]$. Define
\[u(x,t) := \sum_{m\in\mathbb{Z}^d} e^{i|m|^2t} \widehat{f}(m)e^{im\cdot x},\]
and 
\[H_N (x,t) := \sum_{N < \max_i\{|m_i|\} \leq 2N} e^{it|m|^2}e^{im\cdot x}.\]
Then for almost all $t$, $u (t,x) \in C^{s-}$ and $\dim_t(f)\leq (d+1) - s$.
\end{theorem}
Before moving on to the proof of the theorem, we will need the following proposition that easily follows from factorization and the $1$ dimensional Weyl bound.
\begin{proposition}\label{Proposition: Weyl d Torus}
For $N \geq 1$ dyadic, and almost every $t$:
\[\sup_{x\in\mathbb{T}^d} \bigg|\sum_{N \leq \max\{ |m_1|, \cdots, |m_d|\} < 2N} e^{it|m|^2} e^{im\cdot x}\bigg| \lesssim N^{\frac{d}{2}+}.\]
\end{proposition}
\begin{proof}[Theorem \ref{Theorem: General TOrus}]
This proof is substantially easier than Theorem \ref{Theorem: General Sphere Theorem}. We write
\[
\|P_N(u(\cdot,t)\|_{L^\infty_x} = \|f_N*H_N\|_{L^\infty_x}\lesssim \|f_N\|_{L^1_x}\|H_N(\cdot, t)\|_{L^\infty_x}\lesssim N^{\frac{d}{2}-\left(\frac{d}{2}+s\right)}= N^{-s},
\]
for almost every $t$. It follows by Theorem~\ref{upperS}  that for almost every $t$ we have the fractal dimension of the graph of $u$ is bounded above by $(d+1)-s$. 
\end{proof}
\begin{remark}
The $L^1$ condition assumed above is rather unwieldy and is way too strong to obtain a lower bound. We correct this in the following two theorems, which are more analogous to the one dimensional statements for BV functions.
\end{remark}
We now, for simplicity of statement, restrict ourselves to $d = 2$. It may be desirable to estimate the dimension of the graph under an assumption on the Fourier coefficients, in which case we will need more information about how the Fourier coefficients of $f$ behave. In particular, we'll need to define, for $m = (m_1, m_2)$:
\[
\sigma_m(f) = \widehat{f}(m_1+1, m_2+1) - \widehat{f}(m_1+1, m_2) - \widehat{f}(m_1, m_2+1) + \widehat{f}(m_1,m_2),
\]
as well as 
\begin{align*}
    \sigma^1_m(f) &= \widehat{f}(m_1+1, m_2)-\widehat{f}(m_1, m_2)\\
    \sigma^2_m(f) &= \widehat{f}(m_1, m_2+1) - \widehat{f}(m_1,m_2).
\end{align*}
These terms measure how much $\widehat{f}$ varies near the point $m\in \mathbb{Z}^2$.

With the prior definition out of the way, we find Theorem \ref{Theorem: Td L2} by a direct application of summation-by-parts.
\begin{theorem}\label{Theorem: Td L2}
Let $m\in\mathbb{Z}^2$, $s\in \left(0, 1\right)$, and suppose that the Fourier coefficients of $f$ satisfy, for $1\leq i\leq 2$:
\begin{equation}\label{Equation: Td L2 assumptions}
|\widehat{f}(m)| \lesssim \frac{1}{\langle m \rangle ^{1+s}},\,\,|\sigma^i_m(f)|\lesssim \frac{1}{\langle m\rangle^{2+s}},\,\,|\sigma_m(f)| \lesssim  \frac{1}{\langle m \rangle ^{3+s}}.
\end{equation}

Let\[u(x,t) := \sum_{m\in\mathbb{Z}^d} e^{i|m|^2t} \widehat{f}(m) e^{i m \cdot x},\]
be the solution emanating from $f$.
\begin{itemize}
\item[i)] Then for almost all $t$: $u(x,t) \in C^{(s-1)-}$ and hence $\dim_t(f)  \leq 3 - s$.
\item[ii)] If $u(x,t)$ is continuous in $x$ and $r_0 = sup_r\{r\,:\,f\in H^r\}$, then 
\[\dim_t(f) \geq 3 + s - 2r_0.\] %3 - [2r - (s - 1)]
\end{itemize}
In particular, if $f$ satisfies \eqref{Equation: Td L2 assumptions} with $s = \frac{1}{2}$ and $f\not\in H^{1/2}$, then for almost every time, $t$, we have
\[
\dim_t(f) = \tfrac{5}{2}.
\]
\end{theorem}
\begin{proof}
The first claim follows from a direct application of summation-by-parts (done twice) to the function
\[
u_N(x,t) = \sum_{N < \max_i\{|m_i|\} \leq 2N} e^{i|m|^2 t} e^{i m \cdot x}\widehat{f}(m).
\]
The second follows by the Strichartz estimate \eqref{Equation: Torus Strichartz} and the standard interpolation argument.
\end{proof}
\subsection{Bounded Variation on the d-Torus}

In this subsection we comment on and provide some results relating to higher dimensional generalizations to bounded variation and their applications to estimations of the dimension of the graph of the free solution associated to \eqref{Equation: Linear Schrodinger}. 

For $1\leq i\leq d$ and $\lambda_i\in\mathbb{N}$ we choose $\{x_j^{(i)}\}_{j=1}^{\lambda_i}$ so that
\[
0 = x_{1}^{(i)} < \cdots < x_{\lambda_i}^{(i)} = 1.
\]
We define $\Pi$ to be the collection of all tuples of $d$ such sequences.

With these sequences, we define the difference operators $\Delta_{ij}$ for $1\leq j\leq \lambda_i$ to be
\[
\Delta_{ij}f := f(y_1, \cdots, y_{i-1}, x_{j+1}^{(i)}, y_{i+1}, \cdots, y_d) - f(y_1, \cdots, y_{i-1}, x_{j}^{(i)}, y_{i+1}, \cdots, y_d).
\]
\begin{definition}[Vitali Bounded Variation]
Let $V$ be the space of functions $f$ satisfying 
\[
\sup_\Pi \sum_{1\leq j\leq \lambda_i-1}\left|(\prod_i\Delta_{ij})f\right| < \infty.
\]
\end{definition}

There is also a slight modification on this space, Fr\'echet Bounded Variation.
\begin{definition}[Fr\'echet Bounded Variation]
Let $\epsilon_{i}, \nu_j\in\{-1, 1\}$. Then let $F$ be the space of functions satisfying
\[
\sup_{\Pi, \epsilon_{ij}} \left|\sum_{1\leq j\leq \lambda_i-1}(\prod_i\epsilon_{i}\nu_j\Delta_{ij})f\right| < \infty.
\]
\end{definition}

It's clear that $V\subset F$, but these two spaces do not coincide, \cite{BVDefs}. The main benefit of these spaces is that if $f\in V(\mathbb{T}^d)$ and $g$ is continuous then, we have a generalized Stieltjes integration-by-parts formula of the form 
\begin{equation}\label{Equation: Vitali IBP}
\int_{\mathbb{T}^d} f\,dg = (-1)^d\int_{\mathbb{T}^d} g\,df,
\end{equation}
where $dg$ can morally be thought of as $\partial_{x_1, \cdots, x_d}g$. Similarly, when $g(x_1, \cdots, x_d) = \prod_{i}\eta_i(x_i)$ for continuous $\eta_i$ then (see, for example \cite{FrechetBVIntegral}) for $f\in F(\mathbb{T}^d)$ we again have the formula \eqref{Equation: Vitali IBP}.
 
It's interesting to note that, as a consequence of these formulas, the Fourier coefficients of $f\in F(\mathbb{T}^d)$ satisfy
\[
|\widehat{f}(m)|\lesssim \frac{1}{\prod_{\substack{1\leq i\leq d\\m_i\ne 0}} |m_i|},
\]
which is again analogous to the one dimensional case.

With these definitions we can then show the following theorem.

\begin{theorem}
Let $f\in V(\mathbb{T}^d)\setminus{H^{1/2+}(\mathbb{T}^d)}$ (or $F\setminus H^{1/2+}$). Then for almost every $t$, $\dim_t(f) = d+1/2$.
\end{theorem}
\begin{proof}
We prove this for $f\in V$, but the statement also holds for $f\in F$ by the factorization of the convolution kernel. Let 
\[
\widetilde{H}_N(x, t) = \sum_{N\leq \max_i(|m_i|)< 2N}\frac{e^{it|m|^2+ix\cdot m}}{R(m)}, \text{  where }
\]
\[
R(m) = \prod_{\substack{i=1\\ m_i\ne 0}}^d m_i. 
\]
It follows that $\widetilde{H}_N(x,t)$ is continuous in $x$ for almost every $t$, and hence for almost every $t$ we find by \eqref{Equation: Vitali IBP}:
\[
\int f(y)H_N(x-y)\,dx = (-1)^d\int \tilde{H}_N(x-y)\,df(y),
\]
so that
\begin{align*}
    \|P_N(u)\|_{L^\infty} \lesssim \|\widetilde{H}_N(\cdot ,t)\|_{L^\infty}|df|(\mathbb{T}^d)\lesssim_f N^{-\frac12+},
\end{align*}
where we have again invoked \ref{Proposition: Weyl d Torus} and summation-by-parts. It follows that $\dim_t(f)\leq (d+1)-1/2 = d+1/2$.

For the lower bound, we see that 
\[
\|P_N(u)\|_{L^2}\lesssim \|\widetilde{H}_N(\cdot, t)\|_{L^2}\lesssim  N^{-\frac12+},
\]
so that we find the appropriate level for $u$ is $H^{1/2+}$, motivating the statement of the theorem. Now, noting that $u$ is continuous, we assume that $u\not\in H^{1/2+}$ and interpolate to obtain that $\dim_t(f)\geq (d+1)-1/2 = d+1/2$.
\end{proof}

It's worth noting that the above is not in any sense optimal. Indeed, consider a small cube, $R$, supported in $[0,1]^d$. If the sides are parallel to the coordinate axes then the characteristic function of $R$ will clearly be in $V$ (and hence $F$), but if we slightly rotate the square $R$, then it immediately leaves both classes.

We can say more, however: the characteristic function for any simplex $P\subset \mathbb{T}^2$ has fractal dimension that is exactly $\frac{5}{2}$. 
\begin{theorem}\label{Theorem: Torus 2d POlygon}
Let $P$ be a polygon in $\mathbb{T}^2$, and $\chi_P$ the characteristic function associated to $P$. Then $\dim_t(f) = \frac{5}{2}$ for almost all $t$.
\end{theorem}
\begin{proof}
By triangulation we reduce the problem to considering triangles. By utilizing subtraction we can further reduce the problem to considering right triangles with two sides parallel to the coordinate axis. Without loss of generality we can then just consider a non-degenerate right triangle, $\mathcal{T}$, with vertices $(0, 0)$, $(x_2, 0)$, and $(x_2, y_2)$. 

We readily calculate the nontrivial ($n,m\ne 0$) Fourier coefficients of $\chi_\mathcal{T}$ to be
\begin{equation}
\widehat{\chi_\mathcal{T}}(n,m) = 
    \begin{cases}
    \frac{y_2(e^{2\pi i nx_2}-1)e^{-2\pi i ny_2}}{4\pi^2n^2} & m = -\frac{nx_2}{y_2},\,\,n\ne 0\\
    \frac{y_2(e^{-2\pi i ( x_2, y_2) \cdot (n, m)}-1)}{4\pi^2n(my_2+nx_2)} + \frac{(e^{2\pi imy_2}-1)e^{-2\pi i y_2(m+n)}}{4\pi^2nm} & \mbox{else}.
    \end{cases}
\end{equation}

The ommitted cases are all trivial by one dimensional theory, and so is the second term in the last case. We thus restrict ourselves to considering the term (modulo constants)
\begin{equation*}
    \frac{e^{-i ( x_2, y_2) \cdot (n, m)}}{n(m+n)},
\end{equation*}
which will be sufficient to prove the full result by noting that the only difference in the argument will be nearest integer considerations.

We first let $u = e^{-it\bigtriangleup}\chi_\mathcal{T}$ and consider when $|m| < |n|$, and note that in order to show the upper bound we need to bound the $L^\infty$ norm of 
\begin{align*}
P_N(u) &= \sum_{N< |n|\leq 2N}\frac{e^{ i tn^2 +i(x-x_2)n}}{n}\sum_{-|n| < m < |n|}\frac{e^{itm^2+i(y-y_2)m}}{n+m}\\
&= \sum_{N< |n|\leq 2N}\frac{e^{ 2i tn^2 +i(x-x_2)n -(y-y_2)|n|}}{n}\sum_{0 < h < 2|n|}\frac{e^{ith^2+i(y-y_2-2|n|)h}}{h}\\
&= \sum_{0 < h < 4N}\frac{e^{ith^2+i(y-y_2-2|n|)h}}{h}\sum_{\max\left(N, \frac{h+2N}{3}\right) < |n| \leq 2N}\frac{e^{ 2i tn^2 +i(x-x_2)n -(y-y_2)|n|}}{n}.
\end{align*}
From this we find that 
\[
\|P_N(u)\|_{L^\infty}\lesssim \frac{N^{\frac12+}\log N}{N } = O(N^{-1/2+}),
\]
for almost every $t$ uniformly in $N, x_2, y_2$. The region $|m|\geq |n|$ follows in the exact same manner, and hence $u\in C^{1/2-}$ and $\dim_t(f)\leq \frac{5}{2}$.

As for the lower bound, we note that it's a standard result on $\mathbb{R}^d$ that the characteristic function of a measurable set with positive measure are not in $H^{1/2}.$ These results trivially extend to $\mathbb{T}^d$, and hence we conclude $\chi_P\in H^{1/2-}\setminus{H^{1/2}}$ as well as the lower bound as before.
\end{proof}

The above generalizes to polytopes in $[-1 ,1]^d$. Indeed, Stokes theorem can be used to write the Fourier transform of the characteristic function of a $d$-dimensional polytope as a sum of products, each of which contains $d$ homogeneous algebraic factors of degree $-1$, see Theorem $1$ of \cite{diaz2016fourier}. It then follows using the exact same change of variables as above that the following theorem holds.
\begin{theorem}\label{Theorem: Torus general Polygon}
Let $P$ be a polytope in $\mathbb{T}^d$, and $\chi_P$ the characteristic function associated to $P$. Then for almost every $t$, $\dim_t(\chi_P) = d+\frac{1}{2}.$
\end{theorem}


\begin{thebibliography}{10}


    
    \bibitem[AdCl]{BVProperties}
    C.~R.~Adams and J.~A.~Clarkson.
    \newblock Properties of functions {$f(x,y)$} of bounded variation.
    \newblock {\em Trans. Amer. Math. Soc.}, 36(4): 711--730, 1934.

        \bibitem[AdCl2]{BVPropertiesCorrection}
    C.~R.~Adams and J.~A.~Clarkson.
    \newblock A correction to ``{P}roperties of functions {$f(x,y)$} of bounded
      variation.''.
    \newblock {\em Trans. Amer. Math. Soc.}, 46: 468, 1939.
    
    \bibitem[BaGa]{BarGatError}
    P.~Baratella and L.~Gatteschi.
    \newblock The bounds for the error term of an asymptotic approximation of
      {J}acobi polynomials.
    \newblock In {\em Orthogonal polynomials and their applications ({S}egovia,
      1986)}, volume 1329 of {\em Lecture Notes in Math.}, pages 203--221.
      Springer, Berlin, 1988.
    
    
     
 
    \bibitem[Be]{Be}
    M.~V.~Berry.
    \newblock Quantum fractals in boxes.
    \newblock {\em J. Phys. A: Math. Gen.}, pages 6617--6629, 1996.
    
    \bibitem[BeKl]{BeK1}
    Michael~V. Berry and Susanne Klein.
    \newblock Integer, fractional and fractal talbot effects.
    \newblock {\em J. Mod. Optics}, 43: 2139--2164, 1996.
    
    \bibitem[BLN]{BeLe}
    M.~V.~Berry, Z.~V.~Lewis, and J.~F.~Nye.
    \newblock On the {W}eierstrass-{M}andelbrot fractal function.
    \newblock {\em Proc. Roy. Soc. London A}, 370: 459--484, 1980.
    
    \bibitem[BMS]{BMS}
    M.~V. Berry, I.~Marzoli, and W.~Schleich.
    \newblock Quantum carpets, carpets of light.
    \newblock {\em Physics World}, 14(6): 39--44, 2001.
    
    \bibitem[BoDe]{Bourgainl2}
    J.~Bourgain and C.~Demeter.
    \newblock The proof of the {$l^2$} decoupling conjecture.
    \newblock {\em Ann. of Math. (2)}, 182(1): 351--389, 2015.
    
    \bibitem[BGZ]{BurqStrichartz}
    N.~Burq, P.~G\'{e}rard, and N.~Tzvetkov.
    \newblock Bilinear eigenfunction estimates and the nonlinear {S}chr\"{o}dinger
      equation on surfaces.
    \newblock {\em Invent. Math.}, 159(1): 187--223, 2005.
    
    \bibitem[BGZ2]{BurqMultilinearStrichartz}
      N.~Burq, P.~G\'{e}rard, and N.~Tzvetkov.
    \newblock Multilinear eigenfunction estimates and global existence for the
      three dimensional nonlinear {S}chr\"{o}dinger equations.
    \newblock {\em Ann. Sci. \'{E}cole Norm. Sup. (4)}, 38(2): 255--301, 2005.
    
    \bibitem[ChSa]{chamizo2023quantum}
    F.~Chamizo and O.~Santillan.
    \newblock About the quantum talbot effect on the sphere.
    \newblock {\em arXiv preprint arXiv:2302.11063}, 2023.
    
    \bibitem[ChOl]{chenolv} G.~Chen and P.~J.~Olver. Numerical simulation of nonlinear dispersive quantization. {\it Discrete Contin. Dyn. Syst.}  {  34}    no. 3:991--1008, 2014.
\bibitem[CET]{cet} V.~Chousionis, M.~B.~Erdo\u{g}an, and N. Tzirakis. Fractal solutions of linear and nonlinear dispersive partial differential equations. {\it Proc. Lond. Math. Soc. }  (3)   { 110}: 543--564,   2015.

 
 
 
    \bibitem[Cl]{FrechetBVIntegral}
    J.~A. Clarkson.
    \newblock On double {R}iemann-{S}tieltjes integrals.
    \newblock {\em Bull. Amer. Math. Soc.}, 39(12): 929--936, 1933.
    
    \bibitem[ClAd]{BVDefs}
    J.~A.~Clarkson and C.~R.~Adams.
    \newblock On definitions of bounded variation for functions of two variables.
    \newblock {\em Trans. Amer. Math. Soc.}, 35(4): 824--854, 1933.
    
    \bibitem[DaXu]{DX}
    F.~Dai and Y.~Xu.
    \newblock {\em Approximation Theory and Harmonic Analysis on Spheres and
      Balls}.
    \newblock Springer New York, 2013.
    
    \bibitem[DeJa]{DeJa}
    A.~Deliu and B.~Jawerth.
    \newblock Geometrical dimension versus smoothness.
    \newblock {\em Constr. Approx}, 8: 211--222, 1992.
    
    \bibitem[DLR]{diaz2016fourier}
    R.~Diaz, Q.-N.~Le, and S.~Robins.
    \newblock Fourier transforms of polytopes, solid angle sums, and discrete
      volume.
    \newblock {\em arXiv preprint arXiv:1602.08593}, 2016.
    
    
    \bibitem[EGT]{EGT}    M.~B.~Erdo\u{g}an, T.~B.~G\"urel, and N.~Tzirakis. The derivative nonlinear Schr\"odinger equation on the half line. 
    {\em Ann. Inst. H. Poincare Anal. Non Lineaire} 35: 1947--1973, 2018.
    
    \bibitem[ErSh]{ErdShak}
    M.~B.~Erdo\u{g}an and G.~Shakan.
    \newblock Fractal solutions of dispersive partial differential equations on the
      torus.
    \newblock {\em Selecta Mathematica}  25, Sept. 2019.
    
    \bibitem[ErTz1]{ErTz1}
    M.~B.~Erdo\u{g}an and N.~Tzirakis.
    \newblock Global smoothing for the periodic kdv evolution.
    \newblock {\em Int. Math. Res. Not.}, 2013(20): 4589--4614, 2012.
    
    \bibitem[ErTz2]{ErTz2}
    M.~B.~Erdo\u{g}an and N.~Tzirakis.
    \newblock Talbot effect for the cubic nonlinear {S}chr\"{o}dinger equation on
      the torus.
    \newblock {\em Math. Res. Lett.}, 20(6): 1081--1090, 2013.
    
    \bibitem[ErTz3]{BurakNikos}
    M.~B.~Erdogan and N.~Tzirakis.
    \newblock {\em Dispersive partial differential equations, wellposedness and
      applications}, volume~86 of {\em London Mathematical Society Student Texts}.
    \newblock Cambridge University Press, 2016.
    
    \bibitem[F\"{o}]{LF}
    L.~F\"{o}ldv\'{a}ry.
    \newblock Sine series expansion of associated Legendre functions.
    \newblock {\em Acta Geod Geophys}, 2015.
    
    \bibitem[Ga]{GasperPositivity}
    G.~Gasper.
    \newblock Linearization of the product of Jacobi polynomials. {I}.
    \newblock {\em Canadian J. Math.}, 22:  171--175, 1970.
    
    \bibitem[HaLo]{MR2439212}
    J.~H.~Hannay and A.~Lockwood.
    \newblock The quantum {T}albot effect on a sphere.
    \newblock {\em J. Phys. A}, 41(39):395205, 9, 2008.
    
    \bibitem[Hu]{AH}
    A.~Higuchi.
    \newblock Symmetric tensor spherical harmonics on the {N}-sphere and their
      application to the de sitter group {SO(N},1).
    \newblock {\em Journal of Mathematical Physics}, 28(7): 1553--1566, 1987.
    
    
    \bibitem[HoVe]{HV} F.~de la Hoz and L.~Vega. Vortex filament equation for a regular polygon. , {\it  Nonlinearity } 27, no. 12: 3031--3057, 2014. 




    \bibitem[Hu]{huynh2022study}
    C.~N.~Y.~Huynh.
    \newblock {\em A study of dimension estimates in the context of spherical
      Talbot effect and Besov mappings}.
    \newblock PhD thesis, University of Illinois at Urbana-Champaign, 2022.
    
    \bibitem[KaRo]{KaRo}
    L.~Kapitanski and I.~Rodnianski.
    \newblock Does a quantum particle knows the time?
    \newblock In D.~Hejhal, J.~Friedman, M.C. Gutzwiller, and A.M. Odlyzko,
      editors, {\em Emerging applications of number theory}, volume 109 of {\em IMA
      Volumes in Mathematics and its Applications}, pages 355--371. Springer
      Verlag, New York, 1999.
    
    \bibitem[Mc]{MCCONNELL2022353}
    R.~McConnell.
    \newblock Nonlinear smoothing for the periodic generalized nonlinear
     Schr\"odinger equation.
    \newblock {\em Journal of Differential Equations}, 341: 353--379, 2022.
    
    \bibitem[NPW]{MR2253732}
    F.~Narcowich, P.~Petrushev, and J.~Ward.
    \newblock Decomposition of Besov and Triebel-Lizorkin spaces on the
      sphere.
    \newblock {\em J. Funct. Anal.}, 238(2): 530--564, 2006.
    
    \bibitem[NPW2]{MR2237162}
    F.~J.~Narcowich, P.~Petrushev, and J.~D. Ward.
    \newblock Localized tight frames on spheres.
    \newblock {\em SIAM J. Math. Anal.}, 38(2): 574--594, 2006.
    
    \bibitem[OlF]{OlF}
    F.~W.~J. Olver.
    \newblock {\em Asymptotics and special functions}.
    \newblock Computer Science and Applied Mathematics. Academic Press [Harcourt
      Brace Jovanovich, Publishers], New York-London, 1974.
    

    \bibitem[OlP]{Ol}
    P.~J.~Olver.
    \newblock { \it Dispersive quantization.} 
    \newblock { Amer. Math. Monthly}, 117(7): 599--610, 2010.
    
    \bibitem[OlTs]{OlTs}
    P.~J.~Olver and E.~Tsatis.
    \newblock Points of constancy of the periodic linearized  Korteweg–de Vries
      equation.
    \newblock {\em Proc. R. Soc. A.}, 474, 2018.
    
    \bibitem[Os]{Os1}
    K.~I.~Oskolkov.
    \newblock A class of {I. M. V}inogradov's series and its applications in
      harmonic analysis.
    \newblock In A.~A. Gonchar and E.B. Saff, editors, {\em Progress in
      approximation theory}, volume~19 of {\em Springer Ser. Comput. Math.}, pages
      353--402. Springer, New York, 1992.
    
    \bibitem[OsCh]{OC13} K.~I.~Oskolkov and M.~A.~Chakhkiev. Traces of the discrete Hilbert transform with quadratic phase. (Russian). {\it  Tr. Mat. Inst. Steklova} 280 (2013), Ortogonal'nye Ryady, Teoriya Priblizheniĭ i Smezhnye Voprosy, 255--269; translation in {\it Proc. Steklov Inst. Math.} 280,
     no. 1: 248--262, 2013.
    
    
    \bibitem[Ro]{Ro}
    I.~Rodnianski.
    \newblock Fractal solutions of {S}chr\"{o}dinger equation.
    \newblock {\em Contemp. Math.}, 255: 181--187, 2000.
    
    \bibitem[Sz]{szego}
    G.~Szeg\"{o}.
    \newblock {\em Orthogonal {P}olynomials}.
    \newblock American Mathematical Society Colloquium Publications, Vol. 23.
      American Mathematical Society, New York, 1939.
    
    \bibitem[Ta]{Ta}
    H.~F. Talbot.
   % \newblock 
   Facts related to optical science.
    {\em Philo. Mag.}, 9(IV): 401--407, 1836.
        
    \bibitem[TaM1]{Ta1}  M.~Taylor, {\it Tidbits in Harmonic Analysis}, Lecture Notes, UNC, 1998.
\bibitem[TaM2]{Ta2} M.~Taylor. The Schr\"odinger equation on spheres.  {\it   Pacific J. Math.}    209: 145--155, 2003.

    \bibitem[Ve]{V} L.~Vega. The dynamics of vortex filaments with corners. {\it Commun. Pure Appl. Anal.}  14, no. 4: 1581--1601, 2015.
    
    
    \bibitem[Za]{VitaliBV}
    S.~C.~Zaremba. Some applications of multidimensional integration by parts. {\it 
    Ann. Polon. Math.} 21: 85--96,  1968.

 
  \bibitem[ZWZX]{ZWZX} Y.~Zhang, J.~Wen, S.~N.~Zhu, and M.~Xiao. Nonlinear Talbot effect. {\it  Phys. Rev. Lett.  }   104, 183901 (2010).
  
   \end{thebibliography}
\end{document}